\documentclass[12pt,twoside]{article}

\usepackage{graphicx,graphics, amsthm, amsfonts, amsbsy, amssymb, amsmath, cite, array, hyperref, float, tikz, fancyvrb, enumitem, xcolor}
\usepackage[T1]{fontenc}
\everymath{\displaystyle}
\usepackage[margin=1in]{geometry}
\usepackage{mathrsfs}

\allowdisplaybreaks
\renewenvironment{proof}{{\noindent\bfseries Proof.}}{\qed}

\everymath{\displaystyle}

\newtheorem{theorem}{Theorem}[section]
\newtheorem{remark}[theorem]{Remark}
\newtheorem{example}[theorem]{Example}
\newtheorem{problem}{Problem}
\newtheorem{lemma}[theorem]{Lemma}
\newtheorem{corollary}[theorem]{Corollary}

\newtheorem{proposition}[theorem]{Proposition}

\usetikzlibrary{chains,positioning}
\usepackage{tikz,pgfplots}
\usetikzlibrary{decorations.pathmorphing}

\providecommand{\keywords}[1]
{
	\noindent\textbf{Keywords:} #1
}
\providecommand{\ams}[2]
{
	\noindent\textbf{2020 AMS subject classification:} #2
}

\def\qed{\nolinebreak\hfill\rule{.2cm}{.2cm}\par\addvspace{.5cm}}

\begin{document}
	\title{On the quotient matrix and its distinct eigenvalues}
	
\author{
{\small Bilal Ahmad Rather,\ \ Jianfeng Wang\footnote{Corresponding author.
\newline{\it \hspace*{5mm}Email addresses:} bilalahmadrr@gmail.com (B.R. Rather), jfwang@aliyun.com (J.F. Wang), lzhaoxi@aliyun.com (Z.X. Li).}, \ \ Zhaoxi Li}\\[2mm]
\footnotesize School of Mathematics and Statistics, Shandong University of Technology, Zibo 255049, China}
	\date{}
	
	\maketitle
	
	\begin{abstract}
		Let $M$ be an $n\times n$ matrix partitioned into blocks according to a partition $P=\{C_{1},\dots,C_{\ell}\}$ of its index set, and let $Q$ be the corresponding equitable quotient. We study when $Q$ contains every distinct eigenvalue of $M$, so that the distinct spectral data of a large structured matrix can be recovered from a lower-order quotient. We first correct the distinction between a coarsest equitable partition and a minimum-order quotient that captures all distinct eigenvalues. We then establish an exact block-lifting construction with characteristic polynomial $\chi_M(x)=\chi_Q(x)\prod_i(x-\theta_i)^{|C_i|-1}$, derive multiplicity-design consequences, and give a necessary and sufficient eigenspace criterion for quotient capture. Further results quantify automatic capture through geometric multiplicity and describe exact spectral factors generated by twin cells in graph matrices. Applications to adjacency, weighted adjacency, Laplacian, signless Laplacian, distance, and distance-based Laplacian matrices are validated, including minimal one-cell refinements that recover eigenvalues hidden by collapsed twin classes.\\

	\\
	
	\ams{}{15A18, 15A42, 05C50.}

	\end{abstract}

	\section{Introduction}
	Let $\mathcal{M}_n(\mathbb{F})$ denote the set of all $n \times n$ matrices with entries from a field $\mathbb{F}$. For a matrix $A \in \mathcal{M}_n(\mathbb{F})$, the multiset of its eigenvalues is called the \emph{spectrum} and is denoted by $\sigma(A)$. The \emph{spectral radius} $\rho(A)$ of $A$ is the maximum modulus among its eigenvalues. A matrix is called \emph{non-negative} if all its entries are non-negative. When an eigenvalue $\lambda$ appears as a root of the characteristic polynomial with multiplicity $e$, we write $\lambda^{[e]}$ and say that its algebraic multiplicity is $e$. An eigenvalue is \emph{simple} if its algebraic multiplicity equals one.
	For an eigenvalue $\lambda$ of a matrix $M$, we denote by
	$\operatorname{amult}_M(\lambda)$ the algebraic multiplicity of $\lambda$,
	that is, its multiplicity as a root of the characteristic polynomial
	$\chi_M(x)=\det(xI-M)$.
	We exclude trivial cases such as the zero matrix, diagonal matrices, and triangular matrices, whose spectral properties are straightforward and well understood. The symbol $J_{n}$ denotes the $n\times n$ all-ones matrix, and $I_{n}$ denotes the $n\times n$ identity matrix. Given a set $S$ partitioned into subsets $S_{1}\cup\dots\cup S_{t}$ with $t\geq 2$, an indicator vector $\mathbf{1}_{i}$ is a vector whose entries equal $1$ on the coordinates belonging to $S_{i}$ and $0$ elsewhere.
	We work with matrices in $\mathcal{M}_n(\mathbb{R})$, where $\mathbb{R}$ is the field of real numbers.
	
	Consider an $n\times n$ matrix $M$, presented in block form as
	\begin{equation*}
		M= \left[ \begin{array}{c | c | c | c | c}
			b_{1,1} & b_{1,2} & \cdots & b_{1,\ell-1} & b_{1,\ell} \\
			\hline
			b_{2,1} & b_{2,2} & \cdots & b_{2,\ell-1}& b_{2,\ell} \\
			\hline
			\vdots & \vdots & \ddots & \vdots & \vdots \\
			\hline
			b_{\ell-1,1} & b_{\ell-1,2} & \cdots & b_{\ell-1,\ell-1} & b_{\ell-1,\ell}\\
			\hline
			b_{\ell,1} & b_{\ell,2} & \cdots & b_{\ell,\ell-1} & b_{\ell,\ell}\\
		\end{array}\right]_{n}. \end{equation*}
	The rows and columns of $M$ are partitioned according to a partition
	$P=\{ P_{1},\dots,P_{\ell}\}$ of the index set $\{1,2,\dots,n\}$. The matrix $Q=(q_{ij})_{\ell\times \ell}$ has as its $(i,j)$-th entry the average of the row sums of the block $b_{ij}$ of $M$. If every row of $b_{ij}$ has the same sum, for all $i,j$, the partition is called \emph{equitable} (or \emph{regular}); in that case $q_{ij}$ is the common row sum and $Q$ is called the \emph{equitable quotient matrix} of $M$ (see \cite{BH,heamers,cds}). For further notation and terminology concerning matrices, we refer the reader to \cite{hj,hjl,vilmar}.

\medskip
Throughout the paper, ``minimum-order quotient capturing all distinct eigenvalues'' means an equitable quotient having the fewest cells among equitable partitions whose quotient contains every distinct eigenvalue under discussion. This is deliberately distinguished from the coarsest equitable partition of the matrix, which may have fewer cells but may fail to capture all distinct eigenvalues. This distinction removes an ambiguity that is important in several examples below.
	
	Let $\{s_1, s_2, \dots, s_n\}$ and $\{s'_1, s'_2, \dots, s'_m\}$ be two finite sequences of real numbers with $n > m$. The sequence $\{s'_i\}$ is said to \emph{interlace} the sequence $\{s_i\}$ if
	$$
	s_i \geq s'_i \geq s_{n-m+i}
	$$
	for all $i = 1, 2, \dots, m$. The interlacing is called \emph{tight} if there exists an integer $k$ with $0 \leq k \leq m$ such that
	$$
	s_i = s'_i \quad \text{for } i = 1, 2, \dots, k, \quad \text{and} \quad s_{n-m+i} = s'_i \quad \text{for } i = k+1, \dots, m.
	$$
	When $m = n-1$, the interlacing inequalities reduce to the classical chain
	$$
	s_1 \geq s'_1 \geq s_2 \geq s'_2 \geq \cdots \geq s_{n-1} \geq s'_{n-1} \geq s_n.
	$$
	
	The paper is organized as follows. In Section~\ref{section 2} we review some known results and present new families of matrices whose quotient matrices contain all distinct eigenvalues of $M$. Section~\ref{section 3} is devoted to characterizing matrices with exactly two distinct eigenvalues that are captured by the quotient matrix. Section~\ref{section 5} examines the distinct eigenvalues of quotient matrices arising from graph matrices. Section~\ref{section 6} treats quotient matrices that encode all distinct eigenvalues of Laplacian matrices associated with graphs.
	
	\section{Quotient matrices containing all the distinct eigenvalues of $M$}\label{section 2}
	In this section, we recall several basic results and establish new ones concerning the distinct eigenvalues of the quotient matrix of a parent matrix $M$. In Section~\ref{section 4}, we discuss necessary and sufficient conditions on an equitable partition $\pi$ of $M$ under which the quotient matrix captures every distinct eigenvalue.
	
	The following lemma relates the eigenvalues of $M$ to those of its quotient $Q$.
	\begin{lemma}[\cite{heamers}]\label{quotient matrix eigenvalues}
		Let $M$ be a real symmetric matrix of order $n$, let $P$ be a partition of its index set, and let $S$ be the normalized characteristic matrix of $P$. Then the eigenvalues of the symmetric compression $S^{\mathsf T}MS$ interlace those of $M$. If $P$ is equitable, the partition subspace $\operatorname{im}S$ is $M$-invariant, the interlacing is tight, and the usual row-sum quotient is similar to $S^{\mathsf T}MS$. Consequently every eigenvalue of the equitable quotient is an eigenvalue of $M$.
	\end{lemma}
	
	Quotient matrices play a fundamental role in the spectral analysis of matrices equipped with a prescribed partition. They find numerous applications in the spectral theory of graph matrices. When a partition is not equitable, interlacing provides bounds on the eigenvalues of the original matrix; when the partition is equitable, the quotient yields exact spectral information. Various applications and properties of quotient matrices can be found in \cite{DS,mehreen1,atikELA2018,you}. A natural question arises: can the entire spectrum of $M$ be recovered from its smallest possible equitable quotient matrix? A similar question has been addressed for the spectral radius. For non-negative matrices, a definitive answer is available.
	\begin{lemma}[\cite{atikELA2018,you}]\label{spectral radius of positive matrices}
		For a non-negative matrix $M$, the spectral radius of $M$ coincides with that of any of its equitable quotient matrices $Q$.
	\end{lemma}
	For matrices that are not non-negative, Lemma~\ref{spectral radius of positive matrices} may fail, as the following example illustrates:
	$$
	M=\begin{bmatrix}
		10 & -1 & -1 & -4 \\
		-1 & 10 & -1 & -4 \\
		6 & 6 & -14 & 1 \\
		6 & 6 & 1 & -14
	\end{bmatrix}.
	$$
	The spectrum of $M$ is $\{11,\,5.81025,\,-15,\,-9.81025\}$. The quotient matrix corresponding to the equitable partition $\{\{1,2\},\{3,4\}\}$ is
	$$
	m=\begin{bmatrix}
		9 & -5 \\
		12 & -13
	\end{bmatrix},
	$$
	and its spectrum is $\{5.81025,\,-9.81025\}$. Thus $M$ and its equitable quotient matrix need not share the same spectral radius when $M$ is not non-negative. Further counterexamples appear in \cite{atikELA2018,you,bilalijpam}. In view of Lemma~\ref{spectral radius of positive matrices}, it is natural to ask about the remaining eigenvalues of $M$. When do $M$ and $Q$ share \emph{all} the distinct eigenvalues?
	More precisely, suppose the spectrum of $M$ is $\{\lambda_{1}^{n_{1}},\lambda_{2}^{n_{2}},\dots,\lambda_{t}^{n_{t}}\}$, where $\lambda_1,\dots,\lambda_t$ are distinct, $t\geq 2$, $n_i\geq 1$, and $\sum_{i=1}^{t}n_i=n$. Is it always true that $\lambda_1,\lambda_2,\dots,\lambda_t$ are eigenvalues of $Q$? In this direction, Atik~\cite{atik} posed the following problem.
	\begin{problem}[Atik, \cite{atik}]\label{distinct eigenvalue problem}
		Determine necessary and sufficient conditions on an equitable quotient matrix $Q$ for it to contain all the distinct eigenvalues of $M$.
	\end{problem}

	Problem~\ref{distinct eigenvalue problem} was investigated in depth in \cite{bilalijpam}. The key observation is that the phenomenon depends on the choice of equitable partition: with one partition $Q$ and $M$ may have the same distinct eigenvalues, whereas with another partition this property may be lost. The problem was reformulated in terms of the \emph{smallest} possible equitable partition: from which equitable quotient $Q$ can we read off all the distinct eigenvalues of $M$? The following theorem guarantees that such partitions always exist.
	\begin{theorem}[\cite{bilalijpam}]\label{theorem eqitable quotient matrix contains all the eigenvalues with some partition}
		Let $M$ be a matrix of order $n$. Then the following hold.
		\begin{itemize}
			\item[{\rm (i)}] 
There exists an equitable partition whose quotient matrix $Q$ contains the spectral radius of $M$.
			\item[{\rm (ii)}] 
There exists an equitable partition whose quotient matrix $Q$ contains all the distinct eigenvalues of $M$.
		\end{itemize}
	\end{theorem}
	
	Concerning the relationship between the distinct eigenvalues of an equitable quotient $Q$ and its parent matrix $M$, the following problem was raised in \cite{bilalijpam}.
	\begin{problem}\label{problem distinct 2}
		Characterize those matrices for which the spectrum of the quotient matrix, taken with respect to the \emph{smallest} equitable partition, contains all the distinct eigenvalues of the parent matrix.
	\end{problem}
	
	In relation to Problem~\ref{problem distinct 2}, the following family of matrices was presented in \cite{bilalijpam}. The displayed three-cell equitable quotient captures all distinct eigenvalues; the distinction between this capturing partition and the coarsest equitable partition is discussed immediately afterward.
	\begin{proposition}\label{Qmat contains distinct eigenvalues of M}
		Let $M$ be the matrix defined by
		\begin{equation}\label{matrix containing all distinct eigenvalues}
			M=\left[ \begin{array}{c | c | c c c c c}
				4n-2 & -(2n-2) & -2   & -2 & \dots   & -2 & -2\\
				\hline
				-1   & 4n+1   & -4   & -4 & \dots   & -4 & -4\\
				\hline
				-1   & -2(2n-2) & 8n-7 & -4 & \dots   & -4 & -4\\
				-1   & -2(2n-2) & -4 & 8n-7 & \dots   & -4 & -4\\
				\vdots & \vdots & \vdots & \vdots & \ddots & \vdots & \vdots\\
				-1   & -2(2n-2) & -4 & -4 & \dots   & 8n-7 & -4\\
				-1   & -2(2n-2) & -4 & -4 & \dots   & -4 & 8n-7
			\end{array}
			\right ]_{n+2}.
		\end{equation}
		Then its equitable quotient matrix $Q$ given in~\eqref{Qmat of special M} contains all the distinct eigenvalues of $M$.
		\begin{equation}\label{Qmat of special M}
			Q=\begin{bmatrix}
				4n-2 & -(2n-2) & -2n\\
				-1   & 4n+1    & -4n\\
				-1   & -2(2n-2)& 4n-3
			\end{bmatrix}.
		\end{equation}
	\end{proposition}
	\begin{proof}
	For the partition $P=\{\{1\},\{2\},\{3,\ldots,n+2\}\}$, the block row sums are exactly the entries displayed in $Q$, so $P$ is equitable. On the $n-1$ dimensional subspace
	$$
	U=\{(0,0,z)^{\mathsf T}:z\in\mathbb R^n,\ \mathbf1^{\mathsf T}z=0\},
	$$
	the first two coordinates of $Mx$ vanish and the lower-right block acts as
	$$
	(8n-3)I_n-4J_n,
	$$
	so $Mx=(8n-3)x$ for every $x\in U$. The quotient characteristic polynomial factors as
	$$
	\chi_Q(x)=x\bigl(x-(4n-1)\bigr)\bigl(x-(8n-3)\bigr).
	$$
	The partition-constant subspace has dimension $3$ and is invariant, with restriction represented by $Q$. Together with $U$ it spans $\mathbb R^{n+2}$. Hence
	$$
	\chi_M(x)=x\bigl(x-(4n-1)\bigr)\bigl(x-(8n-3)\bigr)^n,
	$$
	and the distinct eigenvalues of $M$ are exactly $0$, $4n-1$, and $8n-3$, all of which occur in $Q$.
	\end{proof}
	Several noteworthy results related to Problem~\ref{distinct eigenvalue problem} appear in \cite{bilalijpam}, including a discussion of the role played by the partition, a result on enlarging partitions (Theorem~2.7 of~\cite{bilalijpam}), and a family of matrices whose equitable quotient with the smallest partition misses exactly one distinct eigenvalue of the parent matrix (Theorem~2.6 of~\cite{bilalijpam}). Proposition~\ref{Qmat contains distinct eigenvalues of M} also illustrates why two notions of minimality must be separated. Every row of~\eqref{matrix containing all distinct eigenvalues} sums to zero, so the one-cell partition is equitable and its quotient is $(0)_{1\times1}$. Thus the displayed three-cell partition is not the coarsest equitable partition. On the other hand, $M$ has three distinct eigenvalues, so any quotient containing all of them must have order at least $3$. The displayed $3\times3$ quotient matrix in \eqref{Qmat of special M} has minimum possible order among equitable quotients that capture all distinct eigenvalues. \vskip 2mm
	
	Thus the search continues for a matrix $M$ whose equitable quotient $Q$---with respect to the smallest equitable partition---contains all distinct eigenvalues, which would provide a partial answer to Problem~\ref{problem distinct 2}.
	
	Consider the $n\times n$ matrix
	\begin{equation}\label{mat 1}
		M=\begin{bmatrix}
			1 & -aJ_{1\times (n-1)}\\
			a J_{(n-1)\times 1} & bI_{n-1}+a(J_{n-1}-I_{n-1})
		\end{bmatrix}_{n},
	\end{equation}
	where $a$ and $b$ are scalars. This matrix contains both positive and negative entries. Define the following vectors:
	\begin{align*}
		X_{1}&=(0,-1,0,\dots,0,1),\;X_{2}=(0,-1,0,\dots,1,0),\; X_{3}=(0,-1,0,\dots,0,1,0,0),\\
		&\vdots\\
		X_{n-4}&=(0,-1,0,0,1,0,\dots,0, 0),\; X_{n-3}=(0,-1,0,1,0,\dots,0,0),\;X_{n-2}=(0,-1,1,0,\dots,0).
	\end{align*}
	These vectors are linearly independent and satisfy $MX_{i}^{\mathsf{T}}=(b-a)X_{i}^{\mathsf{T}}$ for $i=1,2,\dots,n-2.$ Now consider the trivial partition $P=\{\{1,2,\dots,n\}\}$. The corresponding quotient matrix is not equitable, since the first row sum is $1-(n-1)a$ whereas the second is $b+(n-1)a$, which differ in general. Consider instead the partition $P=\{\{1\},\{2,3,\dots,n\}\}$. The associated quotient matrix is
	$$
	Q=\begin{bmatrix} 1 & -(n-1)a\\ a & b+(n-2)a \end{bmatrix},
	$$
	whose eigenvalues are
	$$
	\frac{1}{2} \left(a n - 2 a + b + 1 \pm \sqrt{(a (-n) + 2 a - b - 1)^2 - 4 \left(a^2 n - a^2 + a n - 2 a + b\right)}\right).
	$$
	The partition $P=\{\{1\},\{2,3,\dots,n\}\}$ is equitable. The $(n-2)$-dimensional zero-sum subspace supported on the second cell is an eigenspace for $b-a$, and the complementary partition subspace is represented by $Q$. Consequently
	$$
	\chi_M(x)=(x-(b-a))^{n-2}\chi_Q(x).
	$$
	Moreover, direct substitution gives
	$$
	\det((b-a)I_2-Q)=a(n-1)(2a-b+1).
	$$
	Hence the transverse eigenvalue $b-a$ is genuinely absent from $\sigma(Q)$ exactly when $a(2a-b+1)\neq0$ (for $n\geq3$). If $a=0$ or $b=2a+1$, it already belongs to the quotient spectrum. Finally, the one-cell partition is equitable exactly when
	$$
	1-(n-1)a=b+(n-1)a,
	$$
	that is, when $b=1-2(n-1)a$. Outside this exceptional relation the displayed two-cell partition has the minimum possible number of cells among equitable partitions.
	
	We now specialize the matrix $M$ in~\eqref{mat 1} by taking $a=2$ and $b=5$, with $n\geq3$, and define
	$$
	M^{\prime}=\begin{bmatrix} 1 & -2J_{1\times (n-1)}\\ 2 J_{(n-1)\times 1} & 5I_{n-1}+2(J_{n-1}-I_{n-1}) \end{bmatrix}_{n\times n}.
	$$
	Let $X = (-1, 1, \ldots, 1)$. Then $M^{\prime}X^{\mathsf{T}}=(2n-1)X^{\mathsf{T}}$, so $2n-1$ is an eigenvalue of $M^{\prime}$ with eigenvector $X$. Likewise, $M^{\prime}X_{i}^{\mathsf{T}}=3X_{i}^{\mathsf{T}}$ for $i=1,2,\dots,n-2$. Since the sum of the eigenvalues equals the trace of $M^{\prime}$, the remaining eigenvalue is $3$. Consequently, $\sigma(M^{\prime})=\{(2n-1)^{[1]},3^{[n-1]}\}$, so $M^{\prime}$ has exactly two distinct eigenvalues.
	
	The trivial partition $P=\{\{1,2,\dots,n\}\}$ does not yield constant row sums and is therefore not equitable. However, the partition $P^{\prime}=\{\{1\},\{2,3,\dots,n\}\}$ is equitable, since $M^{\prime}$ has constant row sums within each block. The equitable quotient matrix is
	$$
	Q^{\prime}=\begin{bmatrix}
		1 & -2 (n-1) \\
		2 & 2 (n-2)+5
	\end{bmatrix},
	$$
	and its eigenvalues are $\{3,2n-1\}$. Hence for $M^{\prime}$ and the smallest equitable partition $P^{\prime}=\{\{1\},\{2,3,\dots,n\}\}$, the quotient $Q^{\prime}$ contains all distinct eigenvalues of $M^{\prime}$. Any other two-cell partition with a singleton cell fails to be equitable. For example, $P''=\{\{1,2,\dots,n-1\},\{n\}\}$ does not produce constant row sums in the last block. Thus we have identified a family of matrices $M^{\prime}$ satisfying the requirement of Problem~\ref{problem distinct 2}. We summarize this in the following proposition.
	\begin{proposition}
		Let $n\geq3$, and let $M_{n}$ be the matrix
		$$
		M=\begin{bmatrix}
			1 & -2J_{1\times (n-1)}\\
			2 J_{(n-1)\times 1} & 5I_{n-1}+2(J_{n-1}-I_{n-1})
		\end{bmatrix}_{n\times n}.
		$$
		Then, with the smallest equitable partition $P=\{\{1\},\{2,3,\dots,n\}\}$, its equitable quotient matrix
		$$
		Q=\begin{bmatrix}
			1 & -2 (n-1) \\
			2 & 2 (n-2)+5
		\end{bmatrix}
		$$
		contains all the distinct eigenvalues of $M$.
	\end{proposition}
	\begin{proof}
	The block row sums give the displayed quotient. Its characteristic polynomial is
	$$
	\chi_Q(x)=(x-3)(x-(2n-1)).
	$$
	The subspace $\{(0,y)^{\mathsf T}:\mathbf1^{\mathsf T}y=0\}$ has dimension $n-2$ and is an eigenspace for $3$. Hence
	$$
	\chi_M(x)=(x-3)^{n-2}\chi_Q(x)=(x-3)^{n-1}(x-(2n-1)).
	$$
	Since $n\geq3$, the two eigenvalues are distinct. The one-cell partition is not equitable because its two row-sum types are $3-2n$ and $2n+3$. Therefore a two-cell equitable partition has minimum possible size, and the displayed quotient captures both distinct eigenvalues.
	\end{proof}

	Consider the block matrix
	$$
	M=
	\begin{bmatrix}
		\tfrac{1}{\alpha}J_a & J_{a\times b}\\
		J_{b\times a} & \alpha J_b
	\end{bmatrix},
	$$
	where $a,b$ are positive integers and $\alpha\in\mathbb{R}\setminus\{0\}$. We observe first that every row belonging to the second block is $\alpha$ times every row belonging to the first block. So,
	$
	\operatorname{rank}(M)=1.
	$
	Since $M$ has order $a+b$, the rank--nullity theorem gives
	$
	\dim\ker M=a+b-1.
	$
	Thus $0$ is an eigenvalue of $M$ with multiplicity $a+b-1$.
	
	More explicitly, let $e_1,\ldots,e_{a+b}$ be the standard basis vectors, a convenient basis of $\ker M$ is
	$
	e_i-e_1,~ 2\leq i\leq a,
	$
	together with
	$
	e_j-\alpha e_1,~ a+1\leq j\leq a+b.
	$
	Indeed, the first row of $M$ has the form
	$$
	\Big(
	\underbrace{\tfrac{1}{\alpha},\ldots,\tfrac{1}{\alpha}}_{a},
	\underbrace{1,\ldots,1}_{b}
	\Big),
	$$
	and each of the above vectors is orthogonal to this row. Since all remaining rows are scalar multiples of the first one, each such vector belongs to $\ker M$. These $a+b-1$ vectors are linearly independent.
	 Let
	$$
	X=
	\Big(
	\underbrace{1,\ldots,1}_{a},
	\underbrace{\alpha,\ldots,\alpha}_{b}
	\Big).
	$$
	Then, for every coordinate in the first block,
	$$
	(MX^{\mathsf T})_i
	=
	\frac{a}{\alpha}+\alpha b,
	$$
	whereas every coordinate in the second block equals
	$$
	a+\alpha^2b
	=
	\alpha\left(\frac{a}{\alpha}+\alpha b\right).
	$$
	Hence
	$$
	MX^{\mathsf T}
	=
	\left(
	\frac{a}{\alpha}+\alpha b
	\right)X^{\mathsf T}
	=
	\left(
	\frac{a+\alpha^2b}{\alpha}
	\right)X^{\mathsf T}.
	$$
	Therefore
	$
	\lambda=\tfrac{a+\alpha^2b}{\alpha}
	$
	is the unique nonzero eigenvalue of $M$. Since $a,b>0$ and $\alpha\neq0$, we have
	$
	a+\alpha^2b>0,
	$
	so this eigenvalue never vanishes. Consequently,
	$$
	\sigma(M)
	=
	\left\{
	0^{[a+b-1]},
	\left(\frac{a+\alpha^2b}{\alpha}\right)^{[1]}
	\right\}.
	$$
	
	The same conclusion follows immediately from the trace, since $\operatorname{rank}(M)=1$ and
	$$
	\operatorname{tr}(M)
	=
	\frac{a}{\alpha}+\alpha b
	=
	\frac{a+\alpha^2b}{\alpha}.
	$$
	 Now consider the two-cell partition
	$$
	P=
	\big\{
	\{1,\ldots,a\},
	\{a+1,\ldots,a+b\}
	\big\}.
	$$
	Every row indexed by the first cell has row sum $\tfrac{a}{\alpha}$ within the first cell and row sum $b$ toward the second cell. Similarly, every row indexed by the second cell has row sum $a$ toward the first cell and row sum $\alpha b$ within the second cell. Hence $P$ is equitable, with quotient matrix
	$
	Q=
	\begin{bmatrix}
		\tfrac{a}{\alpha} & b\\
		a & \alpha b
	\end{bmatrix}.
	$
	Its characteristic polynomial is
	$$
	\begin{aligned}
		\det(\lambda I_2-Q)
		&=
		\left(\lambda-\frac{a}{\alpha}\right)
		(\lambda-\alpha b)-ab\\
		&=
		\lambda^2-
		\left(\frac{a}{\alpha}+\alpha b\right)\lambda=
		\lambda
		\left(
		\lambda-\frac{a+\alpha^2b}{\alpha}
		\right).
	\end{aligned}
	$$
	Thus
	$
	\sigma(Q)
	=
	\left\{
	0,
	\tfrac{a+\alpha^2b}{\alpha}
	\right\},
	$
	which is precisely the set of distinct eigenvalues of $M$.
	 A small qualification is needed concerning the one-cell partition
	$
	P_0=\big\{\{1,\ldots,a+b\}\big\}.
	$
	This partition is equitable exactly when all row sums of $M$ coincide. The two possible row sums are
	$
	\tfrac{a}{\alpha}+b
	$ and $
	a+\alpha b.
	$
	Their equality is equivalent to
	$
	(1-\alpha)(a+\alpha b)=0,
	$
	and therefore $P_0$ is equitable precisely for
	$
	\alpha=1
	$ or $
	\alpha=-\tfrac{a}{b}.
	$
	For all other values of $\alpha$, the one-cell partition is not equitable.
	
	Irrespective of these exceptional cases, a quotient arising from a one-cell partition has only one eigenvalue and therefore cannot contain both distinct eigenvalues of $M$. Since $M$ always has exactly two distinct eigenvalues, the two-cell partition $P$ has the smallest possible number of cells among equitable partitions whose quotient spectrum contains every distinct eigenvalue of $M$.
	
	Accordingly, the family
	$$
	M=
	\begin{bmatrix}
		\tfrac{1}{\alpha}J_a & J_{a\times b}\\
		J_{b\times a} & \alpha J_b
	\end{bmatrix},
	\qquad
	\alpha\neq0,
	$$
	provides a valid family of matrices for which an equitable quotient of minimum possible order captures the complete set of distinct eigenvalues. A general characterization of matrices possessing this property, and hence a complete resolution of Problem~\ref{problem distinct 2}, appears substantially more delicate. Nevertheless, this rank-one family gives a concrete nontrivial class for which the desired phenomenon can be established exactly.
	 \vskip 2mm

\begin{remark}\label{rem:rank-one-check}
There is no interchange between the cell size $a$ and the scalar parameter $\alpha$ in the preceding calculation. For $2\leq i\leq a$, the vector $e_i-e_1$ is annihilated because the first $a$ columns of every row are equal. For $a+1\leq j\leq a+b$, the vector $e_j-\alpha e_1$ is annihilated because the first row has entries $\tfrac{1}{\alpha}$ on the first cell and $1$ on the second. Moreover, for
$$
X=(\underbrace{1,\ldots,1}_{a},\underbrace{\alpha,\ldots,\alpha}_{b}),
$$
the first-cell coordinates of $MX^{\mathsf T}$ equal $\tfrac{a}{\alpha}+\alpha b$, whereas the second-cell coordinates equal $a+\alpha^2b=\alpha(\tfrac{a}{\alpha}+\alpha b)$. Hence
$$
MX^{\mathsf T}=\frac{a+\alpha^2b}{\alpha}X^{\mathsf T}.
$$
Thus the displayed kernel basis and the nonzero eigenpair are correct.
\end{remark}

The following result is about the block lifting with prescribed transverse eigenvalues.
\begin{theorem}\label{thm:block-lifting}
Let $Q=(q_{ij})\in\mathbb{C}^{k\times k}$, let $n_1,\ldots,n_k$ be positive integers, and $n=n_1+\cdots+n_k$. Choose scalars $\theta_1,\ldots,\theta_k$. Define the block matrix $M=(M_{ij})_{i,j=1}^{k}$ by
$$
M_{ij}=\frac{q_{ij}}{n_j}J_{n_i\times n_j}\quad (i\neq j),
\quad
\text{and}
\quad
M_{ii}=\theta_i I_{n_i}+\frac{q_{ii}-\theta_i}{n_i}J_{n_i}.
$$
Then the partition into the $k$ consecutive cells of sizes $n_1,\ldots,n_k$ is equitable with quotient $Q$, and
$$
\chi_M(x)=\chi_Q(x)\prod_{i=1}^{k}(x-\theta_i)^{n_i-1},
$$
or equivalently,
$$
\sigma(M)=\sigma(Q)\uplus\biguplus_{i=1}^{k}\Big\{\theta_i^{[n_i-1]}\Big\},
$$
where $\uplus$ denotes union with algebraic multiplicity.
\end{theorem}

\begin{proof}
Let $C_i$ be the $i$-th cell and let $W$ be the subspace of vectors constant on every $C_i$. Each block $M_{ij}$ has row sum $q_{ij}$, so the partition is equitable and $M|_W$ is represented by $Q$ in the cell-indicator basis. For each $i$, define
$$
U_i=\{x:\operatorname{supp}(x)\subseteq C_i,\ \sum_{v\in C_i}x_v=0\}.
$$
If $x\in U_i$, then every off-diagonal all-ones block annihilates $x$, while $J_{n_i}x=0$. Hence $Mx=\theta_i x$. The spaces $W,U_1,\ldots,U_k$ form a direct sum and
$$
\dim W+\sum_{i=1}^{k}\dim U_i=k+\sum_{i=1}^{k}(n_i-1)=n.
$$
Thus, in a basis adapted to this decomposition, $M$ is block diagonal with one block similar to $Q$ and scalar blocks $\theta_iI_{n_i-1}$. The characteristic-polynomial and spectrum formulas follow.
\end{proof}

\begin{corollary}\label{cor:block-lifting-capture}
Under the hypotheses of Theorem~\ref{thm:block-lifting}, if every $\theta_i$ belongs to the set of distinct eigenvalues of $Q$, then $M$ and $Q$ have exactly the same set of distinct eigenvalues. In particular, the quotient $Q$ captures all distinct eigenvalues of $M$ for arbitrary prescribed cell sizes.
\end{corollary}

\begin{proof}
The factorization in Theorem~\ref{thm:block-lifting} shows that the only eigenvalues of $M$ outside the quotient spectrum can be the numbers $\theta_i$. If each $\theta_i\in\sigma(Q)$, no new distinct eigenvalue is introduced.
\end{proof}

	For a $1\times 1$ matrix $M=(a)$ the statement is trivial. For a $2\times2$ matrix $M=\begin{bmatrix}a&b\\c&d\end{bmatrix}$ with two distinct eigenvalues, any quotient containing both distinct eigenvalues must have order at least $2$, and the discrete partition $\{\{1\},\{2\}\}$ gives the quotient $M$ itself. If $M$ has only one distinct eigenvalue, then the capture problem is vacuous at the level of distinct eigenvalues: a one-cell quotient works exactly when the two row sums are equal, while otherwise the discrete partition is required. The condition for a repeated algebraic eigenvalue is
	$$
	(a-d)^2+4bc=0,
	$$
	but this condition alone does not imply that the one-cell partition is equitable. This distinction is important when discussing minimum-order quotients.
	
	We now turn to $3\times 3$ matrices. Let
	$$
	M= \begin{bmatrix} m_{11} & m_{12} & m_{13} \\  m_{21} & m_{22} & m_{23} \\ m_{31} & m_{32} & m_{33} \end{bmatrix}.
	$$
	If $M$ has three distinct eigenvalues, then $Q=M$ and the statement holds trivially. If $M$ has only one eigenvalue (e.g., the zero matrix, a scalar matrix, or a triangular matrix with equal diagonal entries), we disregard it as trivial. Suppose $M$ has an eigenvalue $\lambda$ of multiplicity $2$, and write $M$ in the block form
	$$
	M= \left[ \begin{array}{c|cc} m_{11} & m_{12} & m_{13} \\ \hline m_{21} & m_{22} & m_{23} \\ m_{31} & m_{32} & m_{33} \end{array} \right],
	$$
	corresponding to the partition $P=\{\{1\},\{2,3\}\}$. For this partition to be equitable, the row sums of each block must be constant. Denoting these constants by $c_{ij}$, the equitable quotient matrix is
	\begin{equation}\label{qmatt}
		Q = \begin{bmatrix} c_{11} & c_{12} \\ c_{21} & c_{22} \end{bmatrix},
	\end{equation}
	where
	\begin{align}\label{system eq 1}
		m_{11} = c_{11},\quad m_{12} + m_{13} = c_{12},\quad m_{21} = m_{31} = c_{21},\quad m_{22} + m_{23} = m_{32} + m_{33} = c_{22}.
	\end{align}
	Each eigenvalue of $Q$ is an eigenvalue of $M$. Let $\alpha\leq \beta$ be the two real distinct eigenvalues of $Q$. We wish to adjust the entries of $M$ so that one of these, say $\alpha$, is the repeated eigenvalue of $M$. The eigenvector corresponding to the third eigenvalue must lie in the subspace orthogonal to the space of vectors that are constant on the cells of the partition. For our partition, this subspace is spanned by $X= (0, 1, -1)^{t}$. We enforce that $X$ is an eigenvector with eigenvalue $\lambda_3 = \alpha$, so the condition $MX=\alpha X$ gives
	\begin{equation}\label{system eq 2}
		m_{12}-m_{13}=0,\quad m_{22}-m_{23}=\alpha,\quad m_{32}-m_{33}=-\alpha.
	\end{equation}
	Solving the system~\eqref{system eq 1}--\eqref{system eq 2} gives
	$$
	m_{12}=m_{13}=\frac{c_{12}}{2},\quad m_{22}=m_{33}=\frac{c_{22}+\alpha}{2},\quad m_{23}=m_{32}=\frac{c_{22}-\alpha}{2}.
	$$
	Thus $M$ can be written as
	\begin{equation}\label{qmat 3}
		M=\begin{bmatrix}
			c_{11} & \frac{c_{12}}{2} & \frac{c_{12}}{2} \\
			c_{21} & \frac{1}{2} \left(\alpha+c_{22}\right) & \frac{1}{2} \left(c_{22}-\alpha\right) \\[8pt]
			c_{21} & \frac{1}{2} \left(c_{22}-\alpha\right) & \frac{1}{2} \left(\alpha+c_{22}\right)
		\end{bmatrix}.
	\end{equation}
	Therefore, the construction in~\eqref{qmat 3} produces a $3\times3$ family for which the order-$2$ equitable quotient $Q$ contains all distinct eigenvalues of $M$. This gives a complete description of this particular lifting pattern, but it should not be read as a classification of all $3\times3$ matrices with the property.
	
	As an illustration, take
	$$
	Q = \begin{bmatrix} 1 & -8 \\ 4 & 13 \end{bmatrix},
	$$
	whose eigenvalues are $\{5,9\}$. With $\alpha=5$, $c_{11}=1$, $c_{12}=-8$, $c_{21}=4$, and $c_{22}=13$, we obtain
	$$
	M=\begin{bmatrix}
		1 & -4 & -4 \\
		4 & 9 & 4 \\
		4 & 4 & 9
	\end{bmatrix},
	$$
	and its eigenvalues are $\{5^{2},9\}$. Thus $Q$ is the smallest possible equitable matrix containing all distinct eigenvalues of $M$.\\	

	To examine uniqueness, consider the other two partitions of size $2$, namely $P_{1}=\{\{1,3\},\{2\}\}$ and $P_{2}= \{\{1,2\},\{3\}\}$. For $P_{2}$, the row sums of the first block are $-3$ and $13$, which differ, and hence $P_{2}$ is not equitable. The same holds for $P_{1}$. Thus among the $\binom{3}{2}=3$ partitions with two cells, the only equitable one is $\{\{1\},\{2,3\}\}$, and its equitable quotient matrix contains all distinct eigenvalues of $M$. This construction can be extended to matrices of order $4,5,6$, and beyond. While the task becomes more involved for larger orders, producing explicit representations analogous to~\eqref{qmat 3} contributes to Problem~\ref{problem distinct 2}.
	
	\section{Characterization of matrices with two distinct eigenvalues in the quotient Matrix}\label{section 3}
	We now extend the preceding construction to $4\times 4$ matrices. Consider a matrix partitioned according to $\pi=\bigl\{\{1\},\{2,3,4\}\bigr\}$:
	$$
	M=\left[ \begin{array}{c|ccc}
		m_{11} & m_{12} & m_{13} & m_{14}\\
		\hline
		m_{21} & m_{22} & m_{23} & m_{24}\\
		m_{31} & m_{32} & m_{33} & m_{34}\\
		m_{41} & m_{42} & m_{43} & m_{44}
	\end{array}\right].
	$$
	The partition $\pi$ is equitable for $M$ precisely when the row sum of each block is constant. Denote these constants by $c_{ij}$, so that the equitable quotient matrix is
	\begin{equation}\label{eq:Q-4x4}
		Q=
		\begin{bmatrix}
			c_{11} & c_{12}\\
			c_{21} & c_{22}
		\end{bmatrix},
	\end{equation}
	with the conditions
	\begin{equation}\label{eq:system-1-4x4}
		\begin{aligned}
			m_{11}=c_{11},&\quad m_{21}=m_{31}=m_{41}=c_{21},\\
			m_{22}+m_{23}+m_{24}=c_{22},&\quad m_{32}+m_{33}+m_{34}=c_{22},\\
			m_{42}+m_{43}+m_{44}=c_{22},&\quad m_{12}+m_{13}+m_{14}=c_{12}.
		\end{aligned}
	\end{equation}
	Since $\pi$ is equitable, every eigenvalue of $Q$ is an eigenvalue of $M$. Let $\alpha<\beta$ be the two real distinct eigenvalues of $Q$. We aim to construct a $4\times 4$ matrix $M$ with $\sigma(M)=\{\alpha^{3},\beta\}$ for which $\pi$ is an equitable partition and whose quotient $Q$ contains both $\alpha$ and $\beta$.
	
	The subspace of vectors constant on the cells of $\pi$ is
	$$
	W=\{(a,b,b,b)^{\mathsf T}: a,b\in\mathbb{R}\}\subset\mathbb{R}^4,
	$$
	spanned by $e_1=(1,0,0,0)^{\mathsf T}$ and $u=(0,1,1,1)^{\mathsf T}$. Its orthogonal complement (with respect to the standard inner product) is
	$$
	W^\perp = \{(0,y,z,w)^{\mathsf T}: y+z+w=0\},
	$$
	which is $2$-dimensional. A convenient basis of $W^\perp$ is
	$$
	X_1=(0,1,-1,0)^{\mathsf T},\qquad X_2=(0,1,0,-1)^{\mathsf T}.
	$$
	To make $\alpha$ the repeated (triple) eigenvalue of $M$, we require $W^\perp$ to be contained in the eigenspace of $\alpha$, that is,
	$$
	M X_1 = \alpha X_1, \qquad M X_2 = \alpha X_2.
	$$
	
	A direct computation of $M X_1=\alpha X_1$ and $M X_2=\alpha X_2$ gives
	$$
	m_{12}-m_{13}=0,\quad m_{22}-m_{23}=\alpha,\quad m_{32}-m_{33}=-\alpha,\quad m_{42}-m_{43}=0,
	$$
	and
	$$
	m_{12}-m_{14}=0,\quad m_{22}-m_{24}=\alpha,\quad m_{32}-m_{34}=0,\quad m_{42}-m_{44}=-\alpha.
	$$
	Combining these with the equitability conditions~\eqref{eq:system-1-4x4}, we obtain
	$$
	m_{12}=m_{13}=m_{14}=\frac{c_{12}}{3}, \quad \text{and}\quad m_{21}=m_{31}=m_{41}=c_{21}.
	$$
	With these values, the $3\times 3$ lower-right block $B=(m_{ij})_{2\le i,j\le 4}$ of $M$ takes the form
	\begin{equation}\label{eq:B-structure}
		B = \frac{1}{3}
		\begin{bmatrix}
			2\alpha + c_{22} & -\alpha + c_{22} & -\alpha + c_{22}\\
			-\alpha + c_{22} & 2\alpha + c_{22} & -\alpha + c_{22}\\
			-\alpha + c_{22} & -\alpha + c_{22} & 2\alpha + c_{22}
		\end{bmatrix}
		= \alpha I_3 + \frac{c_{22}-\alpha}{3}J_3,
	\end{equation}
	where $J_3$ is the $3\times 3$ all-ones matrix.
	Thus $M$ can be written as
	\begin{equation}\label{eq:M-4x4-final}
		M =
		\begin{bmatrix}
			c_{11} & \frac{c_{12}}{3} & \frac{c_{12}}{3} & \frac{c_{12}}{3}\\[6pt]
			c_{21} & \frac{2\alpha+c_{22}}{3} & \frac{c_{22}-\alpha}{3} & \frac{c_{22}-\alpha}{3}\\[6pt]
			c_{21} & \frac{c_{22}-\alpha}{3} & \frac{2\alpha+c_{22}}{3} & \frac{c_{22}-\alpha}{3}\\[6pt]
			c_{21} & \frac{c_{22}-\alpha}{3} & \frac{c_{22}-\alpha}{3} & \frac{2\alpha+c_{22}}{3}
		\end{bmatrix}.
	\end{equation}
	
	As $\pi$ is equitable with quotient $Q$ as in~\eqref{eq:Q-4x4}, $M$ leaves $W$ invariant, and the restriction $M|_W$ is represented by $Q$ in the basis $\{e_1,u\}$. Hence the eigenvalues of $Q$, namely $\alpha$ and $\beta$, are eigenvalues of $M$ with eigenvectors in $W$. Moreover, the block $B$ in~\eqref{eq:B-structure} satisfies
	$$
	B x = \alpha x \quad\text{for all }x\in\mathbb{R}^3\text{ with }x_1+x_2+x_3=0,
	$$
	since $J_3x=0$ on that subspace. This is precisely the image of $W^\perp$ under projection onto the last three coordinates, and by construction $M X_1=\alpha X_1$ and $M X_2=\alpha X_2$, so $W^\perp\subseteq E_\alpha$ with $\dim W^\perp=2$.
	Consequently, $M|_W$ has eigenvalues $\alpha$ and $\beta$ (one eigenvector each in $W$), and $M|_{W^\perp}$ is $\alpha I$ (two independent eigenvectors in $W^\perp$).
	Therefore $\sigma(M)=\{\alpha^{3},\beta\}$, and $\alpha,\beta$ are exactly the eigenvalues of the smallest equitable quotient matrix $Q$, which is of order $2$.
	
	We illustrate the above construction with an example. Let
	$$
	Q=\begin{bmatrix}
		1 & 0\\
		3 & 4
	\end{bmatrix},
	$$
	whose eigenvalues are $\{1,4\}$. We choose $\alpha=1$ (the eigenvalue that will be repeated in $M$) and $\beta=4$ (the simple eigenvalue). Here $c_{11}=1$, $c_{12}=0$, $c_{21}=3$, $c_{22}=4$. Constructing $M$ from $Q$ and $\alpha$:
	$$
	M =
	\begin{bmatrix}
		c_{11} & \dfrac{c_{12}}{3} & \dfrac{c_{12}}{3} & \dfrac{c_{12}}{3}\\[6pt]
		c_{21} & \dfrac{2\alpha+c_{22}}{3} & \dfrac{c_{22}-\alpha}{3} & \dfrac{c_{22}-\alpha}{3}\\[6pt]
		c_{21} & \dfrac{c_{22}-\alpha}{3} & \dfrac{2\alpha+c_{22}}{3} & \dfrac{c_{22}-\alpha}{3}\\[6pt]
		c_{21} & \dfrac{c_{22}-\alpha}{3} & \dfrac{c_{22}-\alpha}{3} & \dfrac{2\alpha+c_{22}}{3}
	\end{bmatrix}.
	$$
	With $\alpha=1$, $c_{11}=1$, $c_{12}=0$, $c_{21}=3$, $c_{22}=4$, we obtain
	$$
	\frac{c_{12}}{3}=0,\quad
	\frac{2\alpha+c_{22}}{3}=\frac{2\cdot 1+4}{3}=2,\quad
	\frac{c_{22}-\alpha}{3}=\frac{4-1}{3}=1.
	$$
	Thus
	$$
	M = \begin{bmatrix}
		1 & 0 & 0 & 0\\
		3 & 2 & 1 & 1\\
		3 & 1 & 2 & 1\\
		3 & 1 & 1 & 2
	\end{bmatrix}.
	$$
	Clearly $Q$ is the equitable quotient of $M$ under the partition $\pi=\bigl\{\{1\},\{2,3,4\}\bigr\}$, and the eigenvalues of $Q$ are $1$ and $4$. We now verify that $1$ has multiplicity $3$ and $4$ has multiplicity $1$. The partition subspace
	$$
	W = \{(a,b,b,b)^{\mathsf T}: a,b\in\mathbb{R}\}
	= \operatorname{span}\{e_1,\;u\},
	\quad
	u=(0,1,1,1)^{\mathsf T},
	$$
	is $M$-invariant, and on $W$ the matrix of $M$ (in the basis $\{e_1,u\}$) is exactly $Q$. Hence $M$ possesses eigenvectors in $W$ for the eigenvalues $1$ and $4$. For $\lambda=4$, an eigenvector of $Q$ is $(0,1)^{\mathsf T}$, giving the eigenvector $u=(0,1,1,1)^{\mathsf T}$ of $M$ with $M u = 4u$. For $\lambda=1$, an eigenvector of $Q$ is $(1,-1)^{\mathsf T}$, giving $x_W = e_1 - u = (1,-1,-1,-1)^{\mathsf T}$, and $M x_W = x_W$. Thus $x_W$ is an eigenvector of $M$ in $W$ with eigenvalue $1$.
	
	The orthogonal complement of $W$,
	$$
	W^\perp = \{(0,y,z,w)^{\mathsf T} : y+z+w=0\},
	$$
	is spanned by $X_1 = (0,1,-1,0)^{\mathsf T}$ and $X_2 = (0,1,0,-1)^{\mathsf T}$. A direct calculation shows $M X_1 = X_1$ and $M X_2 = X_2$, so $X_1$ and $X_2$ are eigenvectors of $M$ for the eigenvalue $1$, lying in $W^\perp$.
	Therefore $M$ has one eigenvector $u$ for $\lambda=4$ and three linearly independent eigenvectors $x_W$, $X_1$, $X_2$ for $\lambda=1$. The spectrum of $M$ is $\{1^3,4\}$, and the smallest equitable quotient $Q$ contains both distinct eigenvalues $1$ and $4$ of $M$.
	
	The second possible spectrum for $M$ is $\{\alpha^{[2]},\beta^{[2]}\}$. We seek a $Q$ with the smallest possible partition such that $\sigma(Q)=\{\alpha,\beta\}$. This case is treated in the following theorem.
	\begin{theorem}\label{thm:4x4-two-eigs}
		Let $Q$ be a real $2\times 2$ matrix $Q=
		\begin{bmatrix}
			c_{11} & c_{12}\\
			c_{21} & c_{22}
		\end{bmatrix}$ with two distinct real eigenvalues $\sigma(Q)=\{\alpha,\beta\}$, $\alpha\neq\beta$.
		Define the $4\times 4$ matrix $M$ (depending on $Q,\alpha,\beta$) by
		\begin{equation}\label{eq:M-4x4}
			M=
			\begin{bmatrix}
				\displaystyle \frac{c_{11}+\alpha}{2} & \displaystyle \frac{c_{11}-\alpha}{2} & \displaystyle \frac{c_{12}}{2} & \displaystyle \frac{c_{12}}{2}\\[8pt]
				\displaystyle \frac{c_{11}-\alpha}{2} & \displaystyle \frac{c_{11}+\alpha}{2} & \displaystyle \frac{c_{12}}{2} & \displaystyle \frac{c_{12}}{2}\\[8pt]
				\displaystyle \frac{c_{21}}{2} & \displaystyle \frac{c_{21}}{2} & \displaystyle \frac{c_{22}+\beta}{2} & \displaystyle \frac{c_{22}-\beta}{2}\\[8pt]
				\displaystyle \frac{c_{21}}{2} & \displaystyle \frac{c_{21}}{2} & \displaystyle \frac{c_{22}-\beta}{2} & \displaystyle \frac{c_{22}+\beta}{2}
			\end{bmatrix}.
		\end{equation}
		Then $Q$ contains both distinct eigenvalues of $M$ with the smallest equitable partition $\pi=\bigl\{\{1,2\},\{3,4\}\bigr\}$.
	\end{theorem}
	\begin{proof}
		Let $M$ be the matrix of order $4$ and let $\pi=\{\{1,2\},\{3,4\}\}$ be the equitable partition with corresponding quotient $Q$. The characteristic matrix of $\pi$ is
		$$
		P=\begin{bmatrix}
			1 & 0\\
			1 & 0\\
			0 & 1\\
			0 & 1
		\end{bmatrix},
		$$
		whose columns span the subspace of vectors constant on each cell:
		$$
		W = \operatorname{im}P = \{(a,a,b,b)^\mathsf{T} : a,b\in\mathbb{C}\}.
		$$
		The equitable quotient satisfies $MP = PQ$. The sum of the entries in rows $1$ and $2$ over the columns in $\{1,2\}$ is $\tfrac{c_{11}+\alpha}{2}+\tfrac{c_{11}-\alpha}{2}=c_{11}$, while the sum over columns $\{3,4\}$ is $\tfrac{c_{12}}{2}+\tfrac{c_{12}}{2}=c_{12}$. For rows $3$ and $4$, the sum over columns $\{1,2\}$ is $\tfrac{c_{21}}{2}+\tfrac{c_{21}}{2}=c_{21}$, and the sum over columns $\{3,4\}$ is $\tfrac{c_{22}+\beta}{2}+\frac{c_{22}-\beta}{2}=c_{22}$. Hence the row sums from cell $i$ to cell $j$ equal $c_{ij}$, giving
		$$
		MP = \begin{bmatrix} c_{11} & c_{12}\\ c_{11} & c_{12}\\ c_{21} & c_{22}\\ c_{21} & c_{22} \end{bmatrix}
		= P \begin{bmatrix} c_{11} & c_{12}\\ c_{21} & c_{22} \end{bmatrix}
		= PQ.
		$$
		Thus $\pi$ is equitable with quotient $Q$.
		
		Let $v_1 = (1,1,0,0)^\mathsf{T}$ and $v_2 = (0,0,1,1)^\mathsf{T}$ be a basis of $W$, and let $X_1 = (1,-1,0,0)^\mathsf{T}$ and $X_2 = (0,0,1,-1)^\mathsf{T}$ be a basis of $W^\perp$. Then
		$$
		\mathbb{R}^4 = W \oplus W^\perp = \operatorname{span}\{v_1,v_2\} \oplus \operatorname{span}\{X_1,X_2\}.
		$$
		Since $\pi$ is equitable, $W$ is $M$-invariant and $M|_W$ is represented by $Q$ in the basis $\{v_1,v_2\}$:
		$$
		M v_1 = c_{11} v_1 + c_{21} v_2,\qquad
		M v_2 = c_{12} v_1 + c_{22} v_2.
		$$
		Since $\sigma(Q)=\{\alpha,\beta\}$, $M$ has eigenvalues $\alpha$ and $\beta$ with eigenvectors in $W$.
		
		A direct calculation using~\eqref{eq:M-4x4} gives
		\begin{align*}
			M X_1 &=
			\begin{bmatrix}
				\frac{c_{11}+\alpha}{2} & \frac{c_{11}-\alpha}{2} & \frac{c_{12}}{2} & \frac{c_{12}}{2}\\[ 4pt]
				\frac{c_{11}-\alpha}{2} & \frac{c_{11}+\alpha}{2} & \frac{c_{12}}{2} & \frac{c_{12}}{2}\\[4pt]
				\frac{c_{21}}{2} & \frac{c_{21}}{2} & \frac{c_{22}+\beta}{2} & \frac{c_{22}-\beta}{2}\\[4pt]
				\frac{c_{21}}{2} & \frac{c_{21}}{2} & \frac{c_{22}-\beta}{2} & \frac{c_{22}+\beta}{2}
			\end{bmatrix}
			\begin{bmatrix}1\\-1\\0\\0\end{bmatrix}
			= \begin{bmatrix}
				\alpha\\-\alpha\\0\\0
			\end{bmatrix}
			= \alpha X_1,
		\end{align*}
		and similarly $M X_2 = \beta X_2$. Thus $X_1$ and $X_2$ are eigenvectors of $M$ for $\alpha$ and $\beta$, respectively, lying in $W^\perp$. Altogether, $M$ has four linearly independent eigenvectors: two for $\alpha$ and two for $\beta$, so $\sigma(M)=\{\alpha,\alpha,\beta,\beta\}$.
		
		Any nontrivial equitable quotient of $M$ requires a partition with at least two cells. The natural $2$-cell partition $\pi=\{\{1,2\},\{3,4\}\}$ yields a quotient of order $2$, namely $Q$, which already contains all distinct eigenvalues $\alpha,\beta$ of $M$. Any finer equitable partition would produce a quotient of order at least $3$, and hence $Q$ is the smallest (by order) equitable quotient that contains all distinct eigenvalues of $M$.
	\end{proof}
	
	We illustrate Theorem~\ref{thm:4x4-two-eigs} with an example. Let $Q=
	\begin{bmatrix}
		1 & -8\\
		4 & 13
	\end{bmatrix}$, whose eigenvalues are $\{5,9\}$. Thus $\alpha = 5$, $\beta = 9$, $c_{11}=1$, $c_{12}=-8$, $c_{21}=4$, $c_{22}=13$. Using formula~\eqref{eq:M-4x4},
	$$
	M=
	\begin{bmatrix}
		\displaystyle \frac{c_{11}+\alpha}{2} & \displaystyle \frac{c_{11}-\alpha}{2} & \displaystyle \frac{c_{12}}{2} & \displaystyle \frac{c_{12}}{2}\\[8pt]
		\displaystyle \frac{c_{11}-\alpha}{2} & \displaystyle \frac{c_{11}+\alpha}{2} & \displaystyle \frac{c_{12}}{2} & \displaystyle \frac{c_{12}}{2}\\[8pt]
		\displaystyle \frac{c_{21}}{2} & \displaystyle \frac{c_{21}}{2} & \displaystyle \frac{c_{22}+\beta}{2} & \displaystyle \frac{c_{22}-\beta}{2}\\[8pt]
		\displaystyle \frac{c_{21}}{2} & \displaystyle \frac{c_{21}}{2} & \displaystyle \frac{c_{22}-\beta}{2} & \displaystyle \frac{c_{22}+\beta}{2}
	\end{bmatrix}.
	$$
	Substituting the values,
	\begin{align*}
		\frac{c_{11}+\alpha}{2} &= \frac{1+5}{2}=3, &
		\frac{c_{11}-\alpha}{2} &= \frac{1-5}{2}=-2, &
		\frac{c_{12}}{2} &= \frac{-8}{2}=-4, &
		\frac{c_{21}}{2} &= \frac{4}{2}=2,\\[4pt]
		\frac{c_{22}+\beta}{2} &= \frac{13+9}{2}=11, &
		\frac{c_{22}-\beta}{2} &= \frac{13-9}{2}=2.
	\end{align*}
	Hence
	$$
	M = \begin{bmatrix}
		3 & -2 & -4 & -4\\
		-2 & 3 & -4 & -4\\
		2 & 2 & 11 & 2\\
		2 & 2 & 2 & 11
	\end{bmatrix}.
	$$
	The partition is $\pi=\bigl\{\{1,2\},\{3,4\}\bigr\}$ and the quotient coincides with $Q$. The eigenvalues of $M$ are $\{5^{[2]},9^{[2]}\}$.
	
	The remaining case for $4\times 4$ matrices concerns the partition $\bigl\{\{1\},\{2\},\{3,4\}\bigr\}$ and three distinct eigenvalues $\{\alpha^{[2]},\beta,\gamma\}$.
	\begin{theorem}\label{thm:4x4-three-eigs}
		Let $Q=(c_{ij})_{1\leq i,j\leq3}$ be a real $3\times3$ matrix with three distinct real eigenvalues $\{\alpha,\beta,\gamma\}$. Define $M$ by
		\begin{equation}\label{eq:4x4-M-form}
			M=
			\begin{bmatrix}
				c_{11} & c_{12} & \dfrac{c_{13}}{2} & \dfrac{c_{13}}{2}\\[6pt]
				c_{21} & c_{22} & \dfrac{c_{23}}{2} & \dfrac{c_{23}}{2}\\[6pt]
				c_{31} & c_{32} & \dfrac{c_{33}+\alpha}{2} & \dfrac{c_{33}-\alpha}{2}\\[6pt]
				c_{31} & c_{32} & \dfrac{c_{33}-\alpha}{2} & \dfrac{c_{33}+\alpha}{2}
			\end{bmatrix},
		\end{equation}
		Then $\pi=\bigl\{\{1\},\{2\},\{3,4\}\bigr\}$ is equitable with quotient $Q$, and
		$
		\sigma(M)=\{\alpha^{[2]},\beta,\gamma\}.
		$
		Consequently $Q$ has minimum possible order among equitable quotients that contain all three distinct eigenvalues of this constructed matrix.
	\end{theorem}
	\begin{proof}
		Starting from $Q$ and $\alpha$, consider the partition $\pi=\bigl\{\{1\},\{2\},\{3,4\}\bigr\}$ and write $M$ in block form as
		$$
		M=
		\left[ \begin{array}{c|c|cc}
			m_{11} & m_{12} & m_{13} & m_{14}\\
			\hline
			m_{21} & m_{22} & m_{23} & m_{24}\\
			\hline
			m_{31} & m_{32} & m_{33} & m_{34}\\
			m_{41} & m_{42} & m_{43} & m_{44}
		\end{array} \right].
		$$
		For $\pi$ to be equitable, the row sums over each column-block must be constant within each row-block. Denoting these constants by $c_{ij}$ (with $i$ indexing row-blocks and $j$ column-blocks), equitability yields
		\begin{equation}\label{eq:eq-4x4-3}
			\begin{aligned}
				c_{11}&=m_{11}, & c_{12}&=m_{12}, & c_{13}&=m_{13}+m_{14},\\
				c_{21}&=m_{21}, & c_{22}&=m_{22}, & c_{23}&=m_{23}+m_{24},\\
				m_{31}&=m_{41}=c_{31}, & m_{32}&=m_{42}=c_{32}, & m_{33}+m_{34}&=m_{43}+m_{44}=c_{33}.
			\end{aligned}
		\end{equation}
		Thus the equitable quotient matrix is
		\begin{equation}\label{eq:Q-4x4-3}
			Q=
			\begin{bmatrix}
				c_{11} & c_{12} & c_{13}\\
				c_{21} & c_{22} & c_{23}\\
				c_{31} & c_{32} & c_{33}
			\end{bmatrix}.
		\end{equation}
		The characteristic matrix of the partition is
		$$
		P=\begin{bmatrix}
			1 & 0 & 0\\
			0 & 1 & 0\\
			0 & 0 & 1\\
			0 & 0 & 1
		\end{bmatrix},
		$$
		so the partition subspace is
		$$
		W = \operatorname{im}P = \{(x_1,x_2,y,y)^{\mathsf T}: x_1,x_2,y\in\mathbb{R}\}.
		$$
		The quotient $Q$ is the matrix of $M|_W$ in the basis given by the columns of $P$, so the eigenvalues of $Q$ are exactly the eigenvalues of $M$ arising from $M|_W$.
		
		The one-dimensional complement of $W$ that accommodates the repeated eigenvalue is spanned by $X = (0,0,1,-1)^{\mathsf T}$, which is orthogonal to $W$ in the sense that the sum of its entries on $\{3,4\}$ is zero. Imposing $M X = \alpha X$ gives
		$$
		MX =
		\begin{bmatrix}
			m_{13} - m_{14}\\
			m_{23} - m_{24}\\
			m_{33} - m_{34}\\
			m_{43} - m_{44}
		\end{bmatrix}
		= \alpha
		\begin{bmatrix}
			0\\0\\1\\-1
		\end{bmatrix},
		$$
		hence
		\begin{equation}\label{eq:MX-alphaX}
			m_{13}-m_{14}=0,\quad
			m_{23}-m_{24}=0,\quad
			m_{33}-m_{34}=\alpha,\quad
			m_{43}-m_{44}=-\alpha.
		\end{equation}
		
		Combining~\eqref{eq:eq-4x4-3} and~\eqref{eq:MX-alphaX}: from $m_{13}-m_{14}=0$ and $m_{13}+m_{14}=c_{13}$ we get $m_{13}=m_{14}=\tfrac{c_{13}}{2}$; from $m_{23}-m_{24}=0$ and $m_{23}+m_{24}=c_{23}$ we get $m_{23}=m_{24}=\tfrac{c_{23}}{2}$. For the $(3,4)$ block,
		$$
		m_{31}=m_{41}=c_{31},\quad m_{32}=m_{42}=c_{32},\quad m_{33}+m_{34}=c_{33},\quad m_{43}+m_{44}=c_{33},
		$$
		and
		$$
		m_{33}-m_{34}=\alpha,\quad m_{43}-m_{44}=-\alpha.
		$$
		Solving these gives $m_{33}=\tfrac{c_{33}+\alpha}{2}$, $m_{34}=\tfrac{c_{33}-\alpha}{2}$, $m_{43}=\tfrac{c_{33}-\alpha}{2}$, $m_{44}=\tfrac{c_{33}+\alpha}{2}$.
		
		Thus the matrix defined in~\eqref{eq:4x4-M-form} satisfies the equitability equations~\eqref{eq:eq-4x4-3}. By construction, $W$ is $M$-invariant and $M|_W$ is represented by $Q$. Hence the eigenvalues contributed by $W$ are $\{\alpha,\beta,\gamma\}$. Moreover, $MX=\alpha X$ with $X\notin W$ supplies the second copy of $\alpha$. Consequently $\sigma(M)=\{\alpha^{[2]},\beta,\gamma\}$, and all three distinct eigenvalues appear in $Q$. Since $Q$ is $3\times 3$, it is the smallest possible equitable quotient containing three distinct eigenvalues.
	\end{proof}
	
	\begin{example}\label{ex:4x4-nontrivial}
		We illustrate Theorem~\ref{thm:4x4-three-eigs} with an example. Let
		$$
		Q = \begin{bmatrix}
			5 & -3 & 5\\
			0 &  2 & 5\\
			0 &  0 & 7
		\end{bmatrix}
		$$
		be a matrix with eigenvalues $\{5,2,7\}$. Thus $\alpha = 5$, $\beta=2$, $\gamma=7$, and we construct a matrix $M_4$ with $\sigma(M)=\{5^{[2]},2,7\}$ for which $Q$ is a minimum-order equitable quotient capturing all three distinct eigenvalues. Comparing entries, $c_{11}=5$, $c_{12}=-3$, $c_{13}=5$, $c_{21}=0$, $c_{22}=2$, $c_{23}=5$, $c_{31}=0$, $c_{32}=0$, $c_{33}=7$. By Theorem~\ref{thm:4x4-three-eigs},
		$$
		M=
		\begin{bmatrix}
			c_{11} & c_{12} & \dfrac{c_{13}}{2} & \dfrac{c_{13}}{2}\\[6pt]
			c_{21} & c_{22} & \dfrac{c_{23}}{2} & \dfrac{c_{23}}{2}\\[6pt]
			c_{31} & c_{32} & \dfrac{c_{33}+\alpha}{2} & \dfrac{c_{33}-\alpha}{2}\\[6pt]
			c_{31} & c_{32} & \dfrac{c_{33}-\alpha}{2} & \dfrac{c_{33}+\alpha}{2}
		\end{bmatrix}.
		$$
		Substituting the values and $\alpha=5$ gives
		$$
		M = \begin{bmatrix}
			5 & -3 & 2.5 & 2.5\\
			0 &  2 & 2.5 & 2.5\\
			0 &  0 & 6   & 1  \\
			0 &  0 & 1   & 6
		\end{bmatrix}.
		$$
		Clearly $\pi=\bigl\{\{1\},\{2\},\{3,4\}\bigr\}$ is equitable and its quotient is $Q$ with $\sigma(Q)=\{5,2,7\}$.
	\end{example}
	
	We now generalize the above constructions to matrices of arbitrary order $n$.
	\begin{theorem}\label{thm:n-by-n-two-eigs}
		Let $n\ge 3$ and let $Q=\begin{bmatrix} c_{11} & c_{12} \\ c_{21} & c_{22} \end{bmatrix}$ be a $2\times 2$ matrix with two distinct eigenvalues $\alpha\neq\beta$. Define $M$ by
		\begin{equation}\label{eq:Mn-two-eigs}
			M =
			\begin{bmatrix}
				c_{11} & \dfrac{c_{12}}{\,n-1\,}\,\mathbf{1}_{1\times (n-1)}\\[8pt]
				c_{21}\,\mathbf{1}_{(n-1)\times 1} & \alpha I_{n-1} + \dfrac{c_{22}-\alpha}{\,n-1\,}J_{n-1}
			\end{bmatrix},
		\end{equation}
		where $\mathbf{1}_{r\times s}$ is the $r\times s$ all-ones matrix, $I_{n-1}$ is the identity matrix, and $J_{n-1}$ is the $(n-1)\times(n-1)$ all-ones matrix. Then $\pi=\bigl\{\{1\},\{2,3,\dots,n\}\bigr\}$ is equitable with quotient $Q$ and
		$
		\sigma(M)=\{\alpha^{[n-1]},\beta\}.
		$
		Conversely, among matrices with this equitable partition, the displayed form is forced whenever the transverse subspace
		$$
		\mathcal S_\pi^\perp=\{(0,y)^{\mathsf T}:\mathbf 1^{\mathsf T}y=0\}
		$$
		is contained in $E_\alpha(M)$.
	\end{theorem}
	\begin{proof}
		Write $M\in \mathbb{C}^{n\times n}$ in block form with respect to $\pi=\{\{1\},\{2,\dots,n\}\}$ as
		$$
		M = \begin{bmatrix}
			m_{11} & r^\mathsf{T}\\
			s & B
		\end{bmatrix},
		$$
		where $r,s\in\mathbb{C}^{n-1}$ and $B\in\mathbb{C}^{(n-1)\times(n-1)}$. For $\pi$ to be equitable with quotient $Q=(c_{ij})_{1\le i,j\le 2}$, the row sums from each cell to each cell must be constant.
		
		From cell $C_1=\{1\}$ to $C_1$: $m_{11}=c_{11}$.
		From $C_1$ to $C_2$: $\sum_{j=2}^n m_{1j} = \sum_{j=1}^{n-1} r_j = c_{12}$.
		From $C_2$ to $C_1$: $m_{i1}=c_{21}$ for all $i\ge 2$, so $s=c_{21}\mathbf{1}$.
		From $C_2$ to $C_2$: for each $i\ge 2$, $\sum_{j=2}^n B_{ij} = c_{22}$.
		
		The choices in~\eqref{eq:Mn-two-eigs} give
		$$
		m_{11}=c_{11},\qquad
		r=\frac{c_{12}}{\,n-1\,}\mathbf{1}_{n-1},\qquad
		B = \alpha I_{n-1} + \frac{c_{22}-\alpha}{\,n-1\,}J_{n-1}.
		$$
		Then $\sum_{j=2}^n m_{1j} = \sum_{j=1}^{n-1}\frac{c_{12}}{\,n-1\,} = c_{12}$, and for each $i\ge 2$,
		$$
		m_{i1}=c_{21},\qquad
		\sum_{j=2}^n B_{ij}
		= \alpha + \frac{c_{22}-\alpha}{\,n-1\,}\cdot (n-1)
		= c_{22}.
		$$
		Hence, $\pi$ is equitable with quotient $Q$.
		
		Let $\mathbf{1}=\mathbf{1}_{n-1}$ and $u =\begin{bmatrix}0\\ \mathbf{1}\end{bmatrix}\in\mathbb{C}^{n}$. The partition subspace
		$$
		\mathcal{S}_\pi=\{(a,b,\dots,b)^\mathsf{T}:a,b\in\mathbb{C}\}
		$$
		has orthogonal complement
		$$
		\mathcal{S}_\pi^\perp = \{(0,x_2,\dots,x_n)^\mathsf{T} : x_2+\dots+x_n=0\}
		$$
		of dimension $n-2$. By the standard theory of equitable partitions, $M|_{\mathcal{S}_\pi}$ is represented by $Q$ in the basis $\{e_1,u\}$, so its eigenvalues are $\alpha$ and $\beta$.
		
		Take $x=(0,y)^\mathsf{T}\in\mathcal{S}_\pi^\perp$, so $\mathbf{1}^\mathsf{T}y=0$. Then
		$$
		Mx = \begin{bmatrix}
			c_{11} & r^\mathsf{T}\\
			s & B
		\end{bmatrix}
		\begin{bmatrix}0\\y\end{bmatrix}
		= \begin{bmatrix}
			r^\mathsf{T}y\\
			By
		\end{bmatrix}.
		$$
		Since $r$ is a scalar multiple of $\mathbf{1}$, $r^\mathsf{T}y=0$. Moreover,
		$$
		By = \left(\alpha I_{n-1} + \frac{c_{22}-\alpha}{\,n-1\,}J_{n-1}\right)y
		= \alpha y + \frac{c_{22}-\alpha}{\,n-1\,} J_{n-1}y = \alpha y,
		$$
		because $J_{n-1}y = (\mathbf{1}^\mathsf{T}y)\mathbf{1}=0$. Hence $Mx = \alpha x$. Thus every vector in $\mathcal{S}_\pi^\perp$ is an eigenvector for $\alpha$, contributing $\alpha$ with multiplicity $n-2$. Together with the one copy of $\alpha$ from $\mathcal{S}_\pi$, the total algebraic multiplicity of $\alpha$ is $n-1$, while $\beta$ appears once. Therefore $\sigma(M)=\{\beta,\alpha,\dots,\alpha\}$.
		
		Conversely, assume $\pi=\{\{1\},\{2,\dots,n\}\}$ is equitable with quotient $Q$ and, in addition, $\mathcal S_\pi^\perp\subseteq E_\alpha(M)$. Write $M=\begin{bmatrix}m_{11} & r^\mathsf{T}\\ s & B\end{bmatrix}$. Equitability forces
		$$
		m_{11}=c_{11},\quad \sum_{j=1}^{n-1}r_j=c_{12},\quad s=c_{21}\mathbf{1},\quad \sum_{j=1}^{n-1} B_{ij}=c_{22}\ \text{for all } i.
		$$
		For every $y\perp\mathbf{1}$, the hypothesis gives $M(0,y)^\mathsf{T}=\alpha(0,y)^\mathsf{T}$, and therefore $r^\mathsf{T}y=0$ and $By=\alpha y$. Thus $r=\delta\mathbf{1}$ with $\delta = \frac{c_{12}}{n-1}$. Also $B-\alpha I_{n-1}$ annihilates the codimension-one space $\mathbf1^\perp$, so it has the form $u\mathbf{1}^\mathsf{T}$. The row-sum condition gives $\alpha+(n-1)u_i=c_{22}$ for every $i$, hence $u$ is constant and
		$$
		B=\alpha I_{n-1}+\frac{c_{22}-\alpha}{n-1}J_{n-1}.
		$$
		Therefore the displayed form~\eqref{eq:Mn-two-eigs} is necessary under the stated transverse-eigenspace condition. Without that condition, algebraic multiplicity $n-1$ alone does not force this block form for a general nonsymmetric matrix.
	\end{proof}
	
	The classification of matrices with two distinct eigenvalues as in Theorem~\ref{thm:n-by-n-two-eigs} remains open in general. For instance, if the eigenvalues of $M$ are $\{\alpha^{[i]},\beta^{[n-i]}\}$ with $2\leq i\leq \lfloor n/2\rfloor$ and equitable partition $\pi=\{\{1,2,\dots,i\},\{i+1,\dots,n\}\}$, a similar classification is possible. A corresponding classification for matrices with spectrum $\{\alpha^{[i]},\beta^{[j]},\gamma^{[n-i-j]}\}$ ($t\geq 3$ distinct eigenvalues) is considerably more challenging.
	
	\medskip
	The following consequence is immediate.
	\begin{corollary}\label{cor:multiplicity-design}
Let $Q$ have distinct eigenvalues $\lambda_1,\ldots,\lambda_t$, and apply Theorem~\ref{thm:block-lifting} with cell sizes $n_1,\ldots,n_k$. If $\theta_i\in\{\lambda_1,\ldots,\lambda_t\}$ for every $i$, then $M$ has no distinct eigenvalues outside $\sigma(Q)$, and for every $\lambda_j$,
$$
\operatorname{amult}_M(\lambda_j)=\operatorname{amult}_Q(\lambda_j)+\sum_{\{i:\theta_i=\lambda_j\}}(n_i-1).
$$
Thus arbitrary additional multiplicity can be distributed among the quotient eigenvalues by choosing the cell sizes and assigning each cell a transverse eigenvalue.
\end{corollary}

\begin{proof}
The characteristic-polynomial factorization in Theorem~\ref{thm:block-lifting} is
$$
\chi_M(x)=\chi_Q(x)\prod_i(x-\theta_i)^{n_i-1}.
$$
Collecting the exponent of each factor $x-\lambda_j$ gives the stated formula.
\end{proof}

We now illustrate Theorem~\ref{thm:n-by-n-two-eigs} for a matrix of order $10$.
	\begin{example}\label{ex:n10-two-eigs}
		Take the $2\times 2$ quotient $Q=
		\begin{bmatrix}
			1 & -8\\
			4 & 13
		\end{bmatrix}$, whose eigenvalues are $\{5,9\}$. Here $\beta=9$ and $\alpha=5$ (the eigenvalue to be repeated in $M$). With the partition $\pi=\bigl\{\{1\},\{2,3,\dots,10\}\bigr\}$ and $n=10$, Theorem~\ref{thm:n-by-n-two-eigs} gives
		$$
		M =
		\begin{bmatrix}
			c_{11} & \dfrac{c_{12}}{\,n-1\,}\,\mathbf{1}_{1\times (n-1)}\\[8pt]
			c_{21}\,\mathbf{1}_{(n-1)\times 1} & \alpha I_{n-1} + \dfrac{c_{22}-\alpha}{\,n-1\,}J_{n-1}
		\end{bmatrix}.
		$$
		Substituting $c_{11}=1$, $c_{12}=-8$, $c_{21}=4$, $c_{22}=13$, $\alpha=5$, $n-1=9$,
		$$
		\frac{c_{12}}{n-1} = -\frac{8}{9},\qquad
		\frac{c_{22}-\alpha}{n-1} = \frac{13-5}{9} = \frac{8}{9},
		$$
		we obtain the $10\times 10$ matrix
		$$
		M = \begin{bmatrix}
			1 & -\frac{8}{9} & -\frac{8}{9} & \cdots & -\frac{8}{9}\\
			4 & & & & \\
			4 & & B & & \\
			\vdots & & & & \\
			4 & & & &
		\end{bmatrix},
		$$
		where $B = 5I_9 + \frac{8}{9}J_9$; each diagonal entry is $\frac{53}{9}$ and each off-diagonal entry is $\frac{8}{9}$.
		
		Clearly $\pi=\{\{1\},\{2,\dots,10\}\}$ is equitable with quotient $Q$. The partition subspace $\mathcal{S}_\pi = \{(a,b,\dots,b)^\mathsf{T}:a,b\in\mathbb{C}\}$ has dimension $2$ and is $M$-invariant; the restriction $M|_{\mathcal{S}_\pi}$ is represented by $Q$, contributing eigenvalues $5$ and $9$. The orthogonal complement $\mathcal{S}_\pi^\perp = \{(0,x_2,\dots,x_{10})^\mathsf{T} : x_2+\dots+x_{10}=0\}$ has dimension $8$ and is also $M$-invariant. For any $x=(0,y)^\mathsf{T}\in\mathcal{S}_\pi^\perp$,
		$$
		Mx = \begin{bmatrix}
			\bigl(-\frac{8}{9}\mathbf{1}^\mathsf{T}\bigr)y\\[4pt]
			\bigl(5I_9 + \frac{8}{9}J_9\bigr)y
		\end{bmatrix}
		= \begin{bmatrix}
			0\\ 5y
		\end{bmatrix}
		=5x,
		$$
		since $\mathbf{1}^\mathsf{T}y=0$. Thus $\mathcal{S}_\pi^\perp$ contributes eigenvalue $5$ with multiplicity $8$, and together with the copy from $\mathcal{S}_\pi$ we obtain $\sigma(M)=\{9,\underbrace{5,\dots,5}_{9\text{ times}}\}$. Hence $M$ has exactly two distinct eigenvalues $5$ and $9$, and the smallest nontrivial equitable quotient $Q$ contains both.
	\end{example}
	
	\section{Distinct eigenvalues of quotient matrices}\label{section 4}
	In this section, we establish necessary and sufficient conditions under which an equitable quotient matrix of $M$ captures all distinct eigenvalues. We begin with the $4\times 4$ case.
	\begin{proposition}\label{prop 1}
		Let $M\in\mathbb{C}^{4\times 4}$ have spectrum $\{\lambda_1,\lambda_2,\lambda_2,\lambda_2\}$ with $\lambda_1\neq \lambda_2$, and let $\pi=\{C_1,\dots,C_k\}$ be an equitable partition of $\{1,2,3,4\}$ with characteristic matrix $P$ and equitable quotient matrix $B$, so that $MP = PB$. Let $W=\operatorname{im}P$ be the subspace of vectors constant on each cell of $\pi$, and let $E_\lambda$ denote the eigenspace of $M$ corresponding to $\lambda$.
		\begin{itemize}[nosep]
			\item[{\rm (i)}]
For any $\lambda\in\mathbb{C}$, $\lambda\in\sigma(B)$ if and only if $E_\lambda\cap W\neq\{0\}$. Equivalently, $\lambda$ is an eigenvalue of $B$ precisely when $M$ possesses an eigenvector for $\lambda$ whose entries are constant on each cell of $\pi$.
			\item[{\rm (ii)}] 
In particular, the equitable quotient $B$ contains \emph{both} distinct eigenvalues $\lambda_1$ and $\lambda_2$ of $M$ if and only if $E_{\lambda_1}\cap W\neq\{0\}$ and $E_{\lambda_2}\cap W\neq\{0\}$; that is, $M$ admits at least one eigenvector for each of $\lambda_1,\lambda_2$ that is constant on the cells of $\pi$.
			\item[{\rm (iii)}] 
If, in addition, $\lambda_2$ has geometric multiplicity $3$ (equivalently, $\dim E_{\lambda_2}=3$), then for any equitable partition $\pi$ with at least two cells we automatically have $E_{\lambda_2}\cap W\neq\{0\}$, so $\lambda_2\in\sigma(B)$. In this diagonalizable case, $B$ contains both eigenvalues $\lambda_1,\lambda_2$ if and only if $E_{\lambda_1}\cap W\neq\{0\}$; that is, the unique (up to scaling) eigenvector of $M$ for $\lambda_1$ is constant on each cell of $\pi$.
		\end{itemize}
	\end{proposition}
	
	\begin{proof}
		Let $M\in\mathbb{C}^{4\times 4}$ have spectrum $\{\lambda_1,\lambda_2,\lambda_2,\lambda_2\}$ with $\lambda_1\neq\lambda_2$, and let $\pi=\{C_1,\dots,C_k\}$ be an equitable partition of $\{1,2,3,4\}$ with characteristic matrix $P$ and equitable quotient $B$, so that $MP = PB$. The $j$-th column of $P$ is the indicator vector of $C_j$.
		
		Suppose $\lambda\in\sigma(B)$ and let $y\neq 0$ satisfy $By=\lambda y$. Then
		$$
		M(P y)=P(B y)=\lambda (P y),
		$$
		so $x = P y$ is an eigenvector of $M$ with eigenvalue $\lambda$. By construction, $x$ is a linear combination of indicator vectors of the cells, hence constant on each cell and $x\in W$. Thus $x\in E_\lambda\cap W$, $x\neq 0$, and $E_\lambda\cap W\neq\{0\}$.
		
		Conversely, suppose $x\in E_\lambda\cap W$ with $x\neq 0$. Then $x=Py$ for some $y\in\mathbb{C}^k$, and
		$$
		M P y = M x = \lambda x = \lambda P y.
		$$
		Using $MP=PB$, we get $P(B y - \lambda y)=0$. Since $P$ has full column rank, $By=\lambda y$, so $\lambda\in\sigma(B)$.
		
		Hence $\lambda_1,\lambda_2\in\sigma(B)$ if and only if $E_{\lambda_1}\cap W\neq\{0\}$ and $E_{\lambda_2}\cap W\neq\{0\}$; equivalently, for each $i\in\{1,2\}$ there exists an eigenvector of $M$ for $\lambda_i$ that is constant on each cell of $\pi$.
		
		Now assume $\dim E_{\lambda_2}=3$ and $\dim E_{\lambda_1}=1$, with $E_{\lambda_1}\oplus E_{\lambda_2}=\mathbb{C}^4$. Let $\pi$ have $k\ge 2$ cells (since a $1$-cell partition yields a $1\times1$ quotient, which cannot have two distinct eigenvalues). Then $\dim W=k$, and
		$$
		\dim(E_{\lambda_2}\cap W) \ge \dim E_{\lambda_2} + \dim W - 4 = 3 + k - 4 = k-1 \ge 1,
		$$
		so $E_{\lambda_2}\cap W\neq\{0\}$. Thus $\lambda_2\in\sigma(B)$ for any equitable partition with at least two cells. In this diagonalizable case, the only additional requirement for $B$ to contain both $\lambda_1$ and $\lambda_2$ is $E_{\lambda_1}\cap W\neq\{0\}$; that is, the unique (up to scaling) eigenvector for $\lambda_1$ is constant on each cell of $\pi$.
	\end{proof}
	
	\begin{example}\label{ex:equitable-M}
		Consider
		$$
		M = \begin{bmatrix}
			2 & 0 & 0 & 0 \\
			0 & 2 & 0 & 0 \\
			0 & 0 & \tfrac32 & -\tfrac12 \\
			0 & 0 & -\tfrac12 & \tfrac32
		\end{bmatrix},
		$$
		with $\sigma(M)=\{2,2,2,1\}$. An eigenvector for $\lambda_1=1$ is $v_1=\begin{bmatrix}0\\0\\1\\1\end{bmatrix}$, and eigenvectors for $\lambda_2=2$ are
		$$
		\begin{bmatrix}1\\0\\0\\0\end{bmatrix},\quad
		\begin{bmatrix}0\\1\\0\\0\end{bmatrix},\quad
		\begin{bmatrix}0\\0\\-1\\1\end{bmatrix}.
		$$
		Clearly $\dim E_{\lambda_1}=1$ and $\dim E_{\lambda_2}=3$. We apply Proposition~\ref{prop 1}.
		
		The partition $\pi=\{\{1,2\},\{3,4\}\}$ is equitable. Its characteristic matrix is $P=\begin{bmatrix}
			1 & 0\\
			1 & 0\\
			0 & 1\\
			0 & 1
		\end{bmatrix}$, and the subspace of vectors constant on each cell is
		$$
		W=\operatorname{im}P = \{\,(a,a,b,b)^{\mathsf T}: a,b\in\mathbb{C}\,\}.
		$$
		The equation $MP = PB$ gives $B=\begin{bmatrix} 2 & 0\\ 0 & 1 \end{bmatrix}$, with $\sigma(B)=\{2,1\}=\{\lambda_2,\lambda_1\}$. Thus $B$ contains both distinct eigenvalues of $M$.
		
		For $\lambda_1=1$, the eigenvector $v_1=(0,0,1,1)^{\mathsf T}$ belongs to $W$, so $E_{\lambda_1}\cap W\neq\{0\}$. For $\lambda_2=2$, $\dim E_{\lambda_2}=3$ and $\dim W=2$, giving
		$$
		\dim(E_{\lambda_2}\cap W) \ge 3+2-4=1,
		$$
		so $E_{\lambda_2}\cap W\neq\{0\}$. Hence this example satisfies the hypotheses and conclusion of Proposition~\ref{prop 1}: the equitable quotient $B$ contains both distinct eigenvalues $1$ and $2$ of $M$.
	\end{example}
	
	We now generalize Proposition~\ref{prop 1}.
	\begin{theorem}\label{thm:equitable-all-eigs-general}
		Let $M\in\mathbb{C}^{n\times n}$ have spectrum $\sigma(M) = \{\lambda_1^{\,n_1},\lambda_2^{\,n_2},\dots,\lambda_t^{\,n_t}\}$, where $\lambda_1,\dots,\lambda_t$ are the distinct eigenvalues of $M$ with algebraic multiplicities $n_1,\dots,n_t$. Let $\pi=\{C_1,\dots,C_k\}$ be an equitable partition of $\{1,\dots,n\}$ with characteristic matrix $P\in\mathbb{C}^{n\times k}$ (the $j$-th column of $P$ is the indicator vector of $C_j$), and let $Q\in\mathbb{C}^{k\times k}$ be the corresponding equitable quotient matrix, so that $MP = P Q$. Let $W = \operatorname{im}P \subseteq \mathbb{C}^n$ be the subspace of vectors constant on each cell $C_j$, and let $E_{\lambda_i}$ denote the eigenspace of $M$ corresponding to $\lambda_i$. Then the following hold.
		\begin{enumerate}[nosep]
			\item[{\rm (i)}] 
For any $\lambda\in\mathbb{C}$, $\lambda\in\sigma(Q)$ if and only if $E_\lambda\cap W\neq\{0\}$. Equivalently, $\lambda$ is an eigenvalue of $Q$ precisely when $M$ has an eigenvector for $\lambda$ whose entries are constant on each cell of $\pi$.
			\item[{\rm (ii)}] 
In particular, the equitable quotient $Q$ contains all the distinct eigenvalues $\lambda_1,\dots,\lambda_t$ of $M$ if and only if $E_{\lambda_i}\cap W\neq\{0\}$ for every $i=1,\dots,t$; that is, for each $\lambda_i$ there exists an eigenvector of $M$ for $\lambda_i$ that is constant on each cell of the equitable partition $\pi$.
		\end{enumerate}
	\end{theorem}
	
	\begin{proof}
		Since the cells $C_1,\dots,C_k$ are nonempty and disjoint, the columns of $P$ are linearly independent, so $P$ has full column rank and $\dim W = k$. By definition of an equitable partition, there exists a unique $k\times k$ matrix $Q$ such that $MP = PQ$.
		
		If $\lambda\in\sigma(Q)$, take $y\neq 0$ with $Qy = \lambda y$. Then $M(P y) = P(Q y) = \lambda (P y)$. Set $x = P y$; since $y\neq 0$ and $P$ has full column rank, $x\neq 0$, so $x$ is an eigenvector of $M$ with eigenvalue $\lambda$. By construction $x$ is constant on each cell, so $x\in E_\lambda\cap W$ with $x\neq 0$.
		
		Conversely, if $0\neq x\in E_\lambda\cap W$, write $x=Py$ for some $y\in\mathbb{C}^k$. Then
		$$
		M P y = M x = \lambda x = \lambda P y.
		$$
		Using $MP = PQ$, we get $P(Q y - \lambda y) = 0$, and since $P$ has full column rank, $Qy = \lambda y$ with $y\neq 0$ (because $x\neq 0$). Hence $\lambda\in\sigma(Q)$.
		
		Thus $\lambda_i\in\sigma(Q)$ if and only if $E_{\lambda_i}\cap W\neq\{0\}$ for each $i=1,\dots,t$, which proves both assertions.
	\end{proof}
	
	For an eigenvalue $\lambda$ of a matrix $Q$, we write
	$
	g_Q(\lambda)=\dim\ker(Q-\lambda I)
	$
	for the geometric multiplicity of $\lambda$, equivalently, the dimension
	of the eigenspace of $Q$ corresponding to $\lambda$.
	\begin{theorem}\label{thm:geom-mult-capture}
Under the hypotheses of Theorem~\ref{thm:equitable-all-eigs-general}, let $g_M(\lambda)=\dim E_\lambda(M)$ and $g_Q(\lambda)=\dim E_\lambda(Q)$. Then
$$
g_Q(\lambda)=\dim(E_\lambda(M)\cap W).
$$
Consequently,
$$
g_Q(\lambda)\geq \max\{0,g_M(\lambda)+k-n\}.
$$
In particular, every eigenvalue satisfying $g_M(\lambda)>n-k$ is automatically present in the quotient. If $M$ is diagonalizable, then $Q$ is diagonalizable and the same identity holds for algebraic multiplicities:
$$
\operatorname{amult}_Q(\lambda)=\dim(E_\lambda(M)\cap W).
$$
\end{theorem}

\begin{proof}
The map $y\mapsto Py$ is injective and, by $MP=PQ$, sends $E_\lambda(Q)$ onto $E_\lambda(M)\cap W$. Hence it is an isomorphism between these eigenspaces, proving the first identity. The Grassmann dimension inequality gives
$$
\dim(E_\lambda(M)\cap W)\geq \dim E_\lambda(M)+\dim W-n=g_M(\lambda)+k-n.
$$
If $M$ is diagonalizable, its minimal polynomial is square-free. The restriction $M|_W$, and hence the similar matrix $Q$, has minimal polynomial dividing that of $M$, so $Q$ is also diagonalizable. Its algebraic and geometric multiplicities therefore coincide.
\end{proof}

\begin{corollary}\label{cor:high-multiplicity-eigs}
If an equitable partition has $k$ cells and every distinct eigenvalue of $M$ has geometric multiplicity at least $n-k+1$, then its quotient contains all distinct eigenvalues of $M$.
\end{corollary}

\begin{proof}
For every distinct eigenvalue $\lambda$, the inequality $g_M(\lambda)>n-k$ holds, so Theorem~\ref{thm:geom-mult-capture} applies.
\end{proof}

We illustrate Theorem~\ref{thm:equitable-all-eigs-general} with two examples.
	
	\begin{example}
		Let
	$$
	M = \begin{bmatrix}
		0 & 0 & 1 & 1 & 1\\
		0 & 0 & 1 & 1 & 1\\
		1 & 1 & 0 & 0 & 0\\
		1 & 1 & 0 & 0 & 0\\
		1 & 1 & 0 & 0 & 0
	\end{bmatrix}
	$$
	with partition cells $\{1,2\}$ and $\{3,4,5\}$. The spectrum of $M$ is $\sigma(M)=\{\sqrt{6},\,-\sqrt{6},\,0^{[3]}\}$. Let $\lambda_1=\sqrt{6}$, $\lambda_2=-\sqrt{6}$, and $\lambda_3=0$.
	
	Consider vectors constant on each cell: $v = (a,a,b,b,b)^{\mathsf T}$. Then $Mv = \begin{bmatrix}3b\\3b\\2a\\2a\\2a\end{bmatrix}$. Setting $Mv=\lambda v$ gives $3b = \lambda a$ and $2a = \lambda b$, whence $\lambda^2 = 6$, so $\lambda=\pm\sqrt{6}$. For $\lambda=0$, the eigenspace $E_0$ has dimension $3$, since the trace is zero and we already have two nonzero eigenvalues.
	
	The partition is $\pi: C_1=\{1,2\}$, $C_2=\{3,4,5\}$. The characteristic matrix is $P = \begin{bmatrix}
		1 & 0\\
		1 & 0\\
		0 & 1\\
		0 & 1\\
		0 & 1
	\end{bmatrix}$, and the subspace of vectors constant on each cell is
	$$
	W = \operatorname{im}P = \{(a,a,b,b,b)^{\mathsf T}: a,b\in\mathbb{C}\}.
	$$
	The equitable quotient satisfies $MP = PQ$ with $Q = \begin{bmatrix} 0 & 3\\ 2 & 0 \end{bmatrix}$, whose eigenvalues are $\{\sqrt{6},\,-\sqrt{6}\}$. Thus $Q$ contains exactly the two nonzero eigenvalues $\lambda_1,\lambda_2$ of $M$, but not $0$.
	
	By Theorem~\ref{thm:equitable-all-eigs-general}, $\lambda\in\sigma(Q) \Longleftrightarrow E_\lambda\cap W\neq\{0\}$. For $\lambda_1=\sqrt{6}$ and $\lambda_2=-\sqrt{6}$, we obtain eigenvectors $v=(a,a,b,b,b)^{\mathsf T}\in W$ with $\lambda^2=6$, so $E_{\lambda_1}\cap W\neq\{0\}$ and $E_{\lambda_2}\cap W\neq\{0\}$. Hence $\lambda_1,\lambda_2\in\sigma(Q)$, matching the direct computation.
	
	For $\lambda_3=0$, if $v=(a,a,b,b,b)^{\mathsf T}\in W$ satisfies $Mv=0$, then $3b=0$ and $2a=0$, so $a=b=0$. Thus $E_0\cap W = \{0\}$, and by Theorem~\ref{thm:equitable-all-eigs-general}, $\lambda_3=0$ is not an eigenvalue of $Q$.
	\end{example}

	\begin{example}\label{ex:block-matrix-equitable}
		Consider the $n\times n$ matrix
		$$
		M=\begin{bmatrix}
			1 & -2J_{1\times (n-1)}\\
			2 J_{(n-1)\times 1} & 5I_{n-1}+2(J_{n-1}-I_{n-1})
		\end{bmatrix}_{n\times n},
		$$
		where $J_{r\times s}$ is the all-ones matrix and $I_{m}$ is the $m\times m$ identity. We illustrate Theorem~\ref{thm:equitable-all-eigs-general} for this matrix. Note that
		$$
		5I_{n-1}+2(J_{n-1}-I_{n-1}) = 3I_{n-1}+2J_{n-1}.
		$$
		The matrix $J_{n-1}$ has eigenvalue $n-1$ (with eigenvector $\mathbf{1}$) and eigenvalue $0$ with multiplicity $n-2$. Hence $3I_{n-1}+2J_{n-1}$ has eigenvalue $2n+1$ on the span of $\mathbf{1}$ and eigenvalue $3$ with multiplicity $n-2$.
		
		For any vector $v=\begin{bmatrix}0\\ x\end{bmatrix}$ with $x\in\mathbb{C}^{n-1}$ and $\mathbf{1}^\mathsf{T}x=0$, we have
		$$
		Mv = \begin{bmatrix}
			1\cdot 0 -2J_{1\times(n-1)}x\\
			2J_{(n-1)\times1}\cdot 0 + (3I_{n-1}+2J_{n-1})x
		\end{bmatrix}
		= \begin{bmatrix}0\\ 3x\end{bmatrix}
		=3v.
		$$
		Thus $3$ is an eigenvalue with geometric multiplicity at least $n-2$.
		
		Restrict $M$ to the two-dimensional subspace $U =\operatorname{span}\{e_1,\,u\}\subset\mathbb{C}^n$, where $u = \begin{bmatrix}0\\ \mathbf{1}\end{bmatrix}$. In the basis $\{e_1,u\}$,
		\begin{align*}
			M e_1 &= \begin{bmatrix} 1\\ 2\mathbf{1} \end{bmatrix} = 1\cdot e_1 + 2\cdot u,\\
			M u &= \begin{bmatrix} -2J_{1\times (n-1)}\mathbf{1}\\ (3I_{n-1}+2J_{n-1})\mathbf{1} \end{bmatrix}
			= \begin{bmatrix} -2(n-1)\\ (3+2(n-1))\mathbf{1} \end{bmatrix}
			= -2(n-1)e_1 + (2n+1)u.
		\end{align*}
		Hence $M|_U$ is represented by $A = \begin{bmatrix} 1 & -2(n-1)\\ 2 & 2n+1 \end{bmatrix}$, whose eigenvalues are $2n-1$ and $3$. Therefore
		$$
		\sigma(M)=\{(2n-1)^1,\;3^{\,n-1}\},
		$$
		with eigenvalue $3$ of geometric multiplicity $n-1$.
		
		Consider the equitable partition $\pi=\{C_{1},C_{2}\}$ with $C_1=\{1\}$ and $C_2=\{2,3,\dots,n\}$. The corresponding quotient matrix is $Q = \begin{bmatrix} 1 & -2(n-1)\\ 2 & 2n+1 \end{bmatrix}$, whose spectrum is $\{2n-1,\,3\}$.
		
		By Theorem~\ref{thm:equitable-all-eigs-general}, $\lambda\in\sigma(Q)$ if and only if $E_\lambda\cap W\neq\{0\}$. For $\lambda_1 = 2n-1$, an eigenvector of $Q$ is $(1,-1)^\mathsf{T}$, giving $x_1 = e_1 - u = (1,-1,\dots,-1)^\mathsf{T}\in E_{\lambda_1}\cap W\setminus\{0\}$. For $\lambda_2 = 3$, an eigenvector of $Q$ is $(-(n-1),1)^\mathsf{T}$, giving $x_2 = -(n-1)e_1 + u = (-(n-1),1,\dots,1)^\mathsf{T}\in E_{\lambda_2}\cap W\setminus\{0\}$. Thus $Q$ contains both distinct eigenvalues $2n-1$ and $3$ of $M$, in full agreement with the general criterion.
	\end{example}
	
	\section{Distinct eigenvalues of quotient matrices of graph matrices}\label{section 5}
	Various graph-theoretic matrices encode structural properties of graphs: combinatorial features such as connectivity, diameter, and random walk parameters are reflected in their spectra \cite{cds,BH}. We consider only simple connected graphs $G$ with vertex set $V(G)$ of order $n$ and edge set $E$. The degree of a vertex $v$ is denoted by $d_{v}$, and the set of neighbors of $v$ (excluding $v$ itself) by $N(v)$. The distance $d(u,v)$ is the length of a shortest path between $u$ and $v$. A subset $S\subseteq V(G)$ is \emph{independent} if its vertices are pairwise nonadjacent, and a \emph{clique} if every two distinct vertices are adjacent. The complete graph of order $n$ is $K_{n}$, and the complete bipartite graph is $K_{a,b}$, where $a$ and $b$ are the sizes of the partite sets. A pendent vertex is a vertex of degree one.
	
	The \emph{adjacency matrix} $A(G)$ of $G$ is indexed by the vertices of $G$; its $(i,j)$-th entry is $1$ if $i$ and $j$ are adjacent and $0$ otherwise. Let $\operatorname{diag}(G)$ be the diagonal matrix of vertex degrees. The \emph{Laplacian} and \emph{signless Laplacian} matrices of $G$ are $L(G)=\operatorname{diag}(G)-A(G)$ and $Q(G)=\operatorname{diag}(G)+A(G)$, respectively. The \emph{distance matrix} of a connected graph is $D(G)=(d(u,v))_{n}$. Let $\operatorname{Tr}(G)$ be the diagonal matrix of row sums of $D(G)$. The \emph{distance Laplacian} and \emph{distance signless Laplacian} matrices are $D^{L}(G)=\operatorname{Tr}(G)-D(G)$ and $D^{Q}(G)=\operatorname{Tr}(G)+D(G)$, respectively. Further information on graph matrices can be found in \cite{BH,cds}.

	\begin{theorem}\label{thm:twin}
		Let $M=M^{\mathsf T}\in\mathbb R^{n\times n}$ and let
		$
		\pi=\{C_1,\ldots,C_\eta,S\}
		$
		be an equitable partition, where $S=\{v_1,\ldots,v_r\}$ and $r\geq2$.  Suppose
		$
		M[S,S]=aI_r+b(J_r-I_r)
		$
		and, for each $k$, suppose there is a scalar $c_k$ such that
		$
		M_{xv_i}=c_k
		$
		for all $x\in C_k$ and all $v_i\in S$.  Let
		$
		\theta=a-b.
		$
		Then the following hold.
		\begin{enumerate}[label=\textup{(\roman*)}]
			\item The subspace
			$$
			U_S=\left\{x\in\mathbb R^n:\operatorname{supp}(x)\subseteq S,\ \sum_{v\in S}x_v=0\right\}
			$$
			is contained in $E_\theta(M)$.  In particular,
			$
			\dim E_\theta(M)\geq r-1.
			$			
			\item For any $i$, replacing $S$ by $\{v_i\}$ and $S\setminus\{v_i\}$ produces an equitable refinement $\pi_i$.  Its quotient $Q_i$ contains $\theta$.
			
			\item If $\theta\notin\sigma(Q_\pi)$, then one additional cell is necessary and sufficient to obtain, by splitting $S$, a partition subspace containing a nonzero vector of $U_S$.  Thus $\pi_i$ is minimal among refinements whose new partition-constant direction is required to come from $U_S$.
		\end{enumerate}
	\end{theorem}
	
	\begin{proof}
		If $x\in U_S$, then for $v_k\in S$, we have
		$$
		(Mx)_{v_k}
		=ax_{v_k}+b\sum_{j\neq k}x_{v_j}
		=(a-b)x_{v_k}.
		$$
		For $y\in C_k$, the outside-profile hypothesis gives
		$$
		(Mx)_y=c_k\sum_{j=1}^r x_{v_j}=0.
		$$
		Hence $Mx=\theta x$ on the whole space $U_S$.
		
		Now split $S$ into $\{v_i\}$ and $S\setminus\{v_i\}$.  The internal block has constant row sums between the two new cells, and the outside-profile hypothesis gives constant row sums from every old cell to both new cells.  By symmetry the reverse block row sums are also constant.  Thus the refinement is equitable.
		 The vector $z$ defined by
		$
		z_{v_i}=r-1,~z_{v_j}=-1~(j\neq i),$ and $z_y=0~(y\notin S)
		$
		belongs to $U_S$ and is constant on every cell of the refined partition.  Therefore $Mz=\theta z$ is represented in the quotient, and $\theta\in\sigma(Q_i)$.
		
		Finally, if $S$ is left intact, a partition-constant vector is constant on $S$.  Its intersection with $U_S$ is therefore $\{0\}$.  At least one split of $S$ is necessary to expose a nonzero transverse direction from $U_S$, and the singleton--complement split uses exactly one additional cell.
	\end{proof}
	
	\begin{remark}
		The final clause is deliberately relative.  It does not say that every equitable refinement whose quotient happens to contain $\theta$ must split this particular cell $S$.  Another independent invariant direction elsewhere in the matrix could carry the same eigenvalue.  The multi-cell theorem below handles this issue by grouping equal missing transverse values.
	\end{remark}

	The more general form of Theorem \ref{thm:twin} is given below.	
\begin{theorem}\label{thm:twin-factorization}
	Let $M=M^{\mathsf T}$ and let
	$
	\pi=\{C_1,\ldots,C_k,S_1,\ldots,S_t\}
	$
	be an equitable partition with quotient $Q$.  Suppose every $S_j$ satisfies the hypotheses of Theorem~\ref{thm:twin}, and write
	$
	\theta_j=a_j-b_j.
	$
	Let $R\subseteq\{1,\ldots,t\}$.  For each $j\in R$, split $S_j$ into a singleton and its complement, and leave the remaining cells unchanged.  Let $Q_R$ be the quotient of the resulting refinement.  Then
	$$
	\chi_{Q_R}(x)=\chi_Q(x)\prod_{j\in R}(x-\theta_j).
	$$
	Consequently, if
	$$
	\Theta=\{\theta_j:\theta_j\notin\sigma(Q),\ 1\leq j\leq t\},
	$$
	then among these singleton--complement twin-cell refinements the minimum number of additional cells required to make the quotient contain every value in $\Theta$ is exactly $|\Theta|$.
\end{theorem}

\begin{proof}
	Let $W_\pi$ be the partition-constant subspace of the original partition.  Splitting $S_j$ enlarges this space by one dimension.  The new direction may be taken to be
	$$
	z_j=(|S_j|-1)e_{v}-\sum_{u\in S_j\setminus\{v\}}e_u,
	$$
	which belongs to the zero-sum transverse space of $S_j$ and satisfies
	$
	Mz_j=\theta_j z_j.
	$
	For distinct $j$, the vectors $z_j$ have disjoint supports, and each is orthogonal to $W_\pi$.  Hence, after splitting precisely the cells indexed by $R$, the refined partition subspace is
	$$
	W_\pi\oplus\bigoplus_{j\in R}\operatorname{span}\{z_j\}.
	$$
	Every summand is $M$-invariant.  Since the restriction of $M$ to $W_\pi$ is similar to $Q$, while the restriction to $\operatorname{span}\{z_j\}$ is the scalar $\theta_j$, the displayed characteristic-polynomial factorization follows.
	
	A split therefore introduces exactly its own transverse value, counted with quotient multiplicity.  Equal transverse values need only be introduced once when the objective concerns distinct eigenvalues.  The distinct missing values require distinct splits, proving the minimality statement.
\end{proof}

The following consequence is immediate.
\begin{corollary}\label{cor:false-true-twin-factors}
For an adjacency matrix, a false-twin independent cell $S$ contributes the factor $x^{|S|-1}$, while a true-twin clique cell contributes the factor $(x+1)^{|S|-1}$. For Laplacian and signless Laplacian matrices, a twin cell of common degree $d$ contributes, respectively, the factors $(x-d)^{|S|-1}$ for false twins and $(x-d-1)^{|S|-1}$ or $(x-d+1)^{|S|-1}$ for true twins, in agreement with the eigenvalues derived later in Theorem~\ref{end}.
\end{corollary}

In the spirit of Problem~\ref{problem distinct 2}, the following question arises naturally.
	\begin{problem}\label{problem distinct 4}
		Let $A(G)$ be the adjacency matrix of a graph $G$, and let $Q$ be its equitable quotient with respect to some partition $P$. Suppose $\sigma(A(G))=\{\lambda_{1}^{\alpha_{1}},\dots,\lambda_{k}^{\alpha_{k}}\}$ with $\alpha_i\geq 1$, $\sum_{i=1}^{k}\alpha_i=n$, and the $\lambda_i$ distinct. Characterize those graphs for which the spectrum of the quotient matrix with the smallest possible equitable partition contains all the distinct eigenvalues of $A(G)$.
	\end{problem}
	
	Problem~\ref{problem distinct 4} appears to be difficult in full generality; characterizing all graphs for which the smallest equitable quotient captures every distinct eigenvalue of $A(G)$ is highly nontrivial. Nevertheless, identifying even a single such family is worthwhile. If $G$ is $r$-regular, the coarsest equitable partition (the whole vertex set) yields a $1\times1$ quotient matrix whose sole entry is $r$, so the remaining eigenvalues of $A(G)$ are absent from this trivial quotient. Consequently, a graph with the desired property must be nonregular---a useful starting point.
	
	We now present a family of graphs on $n$ vertices for which the spectrum of the quotient matrix, with the smallest possible equitable partition (that is, without subdividing any cell), contains all distinct eigenvalues of $A(G)$.
	\begin{figure}[H]
		\centering
		\begin{tikzpicture}[scale=1.1]
			\draw[fill=black] (-1,0) circle (2pt); \draw[fill=black] (1,0) circle (2pt); \draw[fill=black] (0,1) circle (2pt); \draw[fill=black] (0,.5) circle (2pt);
			\draw[thin] (-1,0)--(1,0)--(0,1)--(-1,0); \draw[thin] (1,0)--(0,.5);
			
			\draw[fill=black] (1.5,0) circle (2pt); \draw[fill=black] (1.7,.5) circle (2pt); \draw[fill=black] (1.8,1) circle (2pt); \draw[fill=black] (1.8,1.5) circle (2pt); \draw[fill=black] (1.7,2) circle (2pt); \draw[fill=black] (1.5,2.5) circle (2pt); \draw[fill=black] (1.1,3) circle (2pt); \draw[fill=black] (.6,3.3) circle (2pt); \draw[fill=black] (0,3.5) circle (2pt);
			
			\draw[thin] (1.7,.5)--(0,1)--(1.5,0); \draw[thin] (1.8,1)--(0,1)--(1.8,1.5); \draw[thin] (1.7,2)--(0,1)--(1.5,2.5); \draw[thin] (1.1,3)--(0,1)--(.6,3.3); \draw[thin] (-.6,3.3)--(0,1)--(0,3.5);
			
			\draw[fill=black] (-1.5,0) circle (2pt); \draw[fill=black] (-1.7,.5) circle (2pt); \draw[fill=black] (-1.8,1) circle (2pt); \draw[fill=black] (-1.8,1.5) circle (2pt); \draw[fill=black] (-1.7,2) circle (2pt); \draw[fill=black] (-1.5,2.5) circle (2pt); \draw[fill=black] (-1.1,3) circle (2pt); \draw[fill=black] (-.6,3.3) circle (2pt);
			
			\draw[thin] (-1.7,.5)--(0,1)--(-1.5,0); \draw[thin] (-1.8,1)--(0,1)--(-1.8,1.5); \draw[thin] (-1.7,2)--(0,1)--(-1.5,2.5); \draw[thin] (0,1)--(-1.1,3);
		\end{tikzpicture}
		\caption{Graph $G$ with $17$ pendent vertices attached to one vertex of $K_{3}$ and one pendent vertex attached to another.}
		\label{unicyclic}
	\end{figure}
	
	Let $G$ be obtained from $K_{3}$ by attaching $a\geq 2$ pendent vertices to one vertex and $1$ pendent vertex to another. The order of $G$ is $n=a+4$. Denote by $u$ the vertex of degree $a+2$, by $v$ the vertex of degree $3$, by $w$ the vertex of degree $2$, by $v_i$ ($1\leq i\leq a$) the pendent vertices attached to $u$, and by $\{b\}$ the pendent vertex attached to $v$. An illustration with $a=17$ is given in Figure~\ref{unicyclic}.
	
	With this labeling, $\{v_{1},\dots,v_{a}\}$ is an independent set whose vertices share the common neighbor $u$. One readily checks that $0$ is an eigenvalue of $G$ with multiplicity at least $a-1$; eigenvectors can be taken as $X_{i-1}=(0,0,0,0,1,x_{i2},x_{i3},\dots,x_{ia})^{t}$ for $2\leq i \leq a$, where $x_{ij}=-1$ if $i=j$ and $0$ otherwise.
	
	The smallest equitable partition is $\{\{u\},\{v\},\{w\},\{v_{1},\dots,v_{a}\},\{b\}\}$. With respect to this partition, the adjacency matrix and its equitable quotient are
	$$
	A(G)=\left[ \begin{array}{c | c | c | c c c c | c}
		0 & 1 & 1 & 1 & 1 & \dots & 1 & 0 \\
		\hline
		1 & 0 & 1 & 0 & 0 & \dots & 0 & 1\\
		\hline
		1 & 1 & 0 & 0 & 0 &  \dots & 0 & 0\\
		\hline
		1 & 0 & 0 & 0 & 0 & \dots & 0 & 0\\
		\vdots & \vdots & \vdots & \vdots & \vdots & \ddots & \vdots & \vdots \\
		1 & 0 & 0 & 0 & 0 & \dots & 0 & 0\\
		1 & 0 & 0 & 0 & 0 & \dots & 0 & 0\\
		\hline
		0 & 1 & 0 & 0 & 0 & \dots & 0 & 0
	\end{array}\right],\qquad
	Q(G)=\begin{bmatrix}
		0 & 1 & 1 & a & 0\\
		1 & 0 & 1 & 0 & 1\\
		1 & 1 & 0 & 0 & 0\\
		1 & 0 & 0 & 0 & 0\\
		0 & 1 & 0 & 0 & 0
	\end{bmatrix}.
	$$
	The characteristic polynomial of $Q$ is $x \bigl(x^{4}-(a+4) x^{2}-2 x+2 a+1\bigr)$. Hence the quotient contains every distinct eigenvalue of $A(G)$ under the displayed coarsest vertex-type partition. The four nonzero eigenvalues are the roots of the displayed quartic; explicit radical expressions exist in principle but are not useful here.
	
	As a concrete illustration, take $a=2$. The adjacency and quotient matrices become
	$$
	A(G)=\left[ \begin{array}{c | c | c | c c | c}
		0 & 1 & 1 & 1 & 1 & 0 \\
		\hline
		1 & 0 & 1 & 0 & 0 & 1\\
		\hline
		1 & 1 & 0 & 0 & 0 & 0\\
		\hline
		1 & 0 & 0 & 0 & 0 & 0\\
		1 & 0 & 0 & 0 & 0 & 0\\
		\hline
		0 & 1 & 0 & 0 & 0 & 0
	\end{array}\right],\qquad
	Q(G)=\begin{bmatrix}
		0 & 1 & 1 & 2 & 0\\
		1 & 0 & 1 & 0 & 1\\
		1 & 1 & 0 & 0 & 0\\
		1 & 0 & 0 & 0 & 0\\
		0 & 1 & 0 & 0 & 0
	\end{bmatrix}.
	$$
	The eigenvalues of $A(G)$ are $\{2.44579, 0.796815, 0^{2}, -1.37033, -1.87228\}$, and those of $Q$ are $\{2.44579,\,0.796815,\,0,\,-1.37033,\,-1.87228\}$. We now generalize this construction to weighted adjacency matrices.
	
	The \emph{weighted adjacency matrix} $\dot{A}(G)$ of a graph $G$ of order $n$ is a real symmetric matrix defined by
	\begin{equation*}\label{general adjacency matrix}
		\dot{A}(G)=A_{\phi}(G)=(a_{\phi})_{n},\qquad a_{ij}=
		\begin{cases}
			\phi(d_{u},d_{v}) & \text{if } u\sim v,\\
			0 & \text{otherwise},
		\end{cases}
	\end{equation*}
	where $u\sim v$ denotes adjacency. The edge weights are $w_{i,j}= \phi(d_{i},d_{j}) = \phi(d_{j},d_{i}) = w_{j,i}$. For $\phi(d_{i},d_{j})=1$ we recover the classical adjacency matrix $A(G)$; other choices of $\phi$ yield matrices such as the Zagreb matrix, the Sombor matrix, the Geometric--Arithmetic matrix, the ABC matrix and so on.
	
	Let $G$ be obtained from $K_{3}$ by attaching $a\geq 2$ pendent vertices to one vertex and $1$ pendent vertex to another. With the partition $\{\{u\},\{v\},\{w\},\{v_{1},\dots,v_{a}\},\{b\}\}$, the weighted adjacency matrix is
	$$
	\dot{A}(G)=\begin{bmatrix}
		0 & \phi(a+2,3) & \phi(a+2,2) & \phi(a+2,1)J_{1\times a} & 0\\
		\phi(a+2,3) & 0 & \phi(3,2) & 0_{1\times a} & \phi(3,1)\\
		\phi(a+2,2) & \phi(3,2) & 0 & 0_{1\times a} & 0\\
		\phi(a+2,1)J_{a\times 1} & 0_{a\times 1} & 0_{a\times 1} & 0_{a} & 0_{a\times 1}\\
		0 & \phi(3,1) & 0 & 0_{1\times a} & 0
	\end{bmatrix},
	$$
	where $0$ denotes a zero matrix and $J$ the all-ones matrix. As before, $0$ is an eigenvalue of $\dot{A}(G)$ with multiplicity $a-1$, with the same set of eigenvectors $X_{i-1}$ for $2\leq i\leq a$. The equitable quotient matrix is
	$$
	Q(\dot{A}(G)) = \begin{bmatrix}
		0 & \phi(a+2,3) & \phi(a+2,2) & a \phi(a+2,1) & 0 \\
		\phi(a+2,3) & 0 & \phi(3,2) & 0 & \phi(3,1) \\
		\phi(a+2,2) & \phi(3,2) & 0 & 0 & 0 \\
		\phi(a+2,1) & 0 & 0 & 0 & 0 \\
		0 & \phi(3,1) & 0 & 0 & 0
	\end{bmatrix}.
	$$
	The characteristic polynomial of $Q(\dot{A}(G))$ is
	\begin{align*}
		x\Bigl(&x^{4}-x^{2} \bigl(\phi(a+2,3)^2+\phi(a+2,2)^2+a \phi(a+2,1)^2+\phi(3,2)^2+\phi(3,1)^2\bigr)\\
		&\quad -2 \phi(a+2,3) \phi(a+2,2) \phi(3,2) x
		+\phi(a+2,2)^2 \phi(3,1)^2 + a \phi(a+2,1)^2 \phi(3,2)^2\\
		&\quad + a \phi(a+2,1)^2 \phi(3,1)^2\Bigr).
	\end{align*}
	Thus $Q(\dot{A}(G))$ contains all the distinct eigenvalues of $\dot{A}(G)$ under the displayed five-cell equitable partition $\Gamma=\{\{u\},\{v\},\{w\},\{v_{1},\dots,v_{a}\},\{b\}\}$. For special choices of the weight function a coarser equitable partition may exist, so no minimality claim is needed here. We summarize this as follows.
	\begin{proposition}\label{rare graph unicyclic}
		Let $G$ be a graph obtained from $K_{3}$ by attaching $a\geq 2$ pendent vertices to one vertex and $1$ pendent vertex to another. Then the quotient matrix of $G$ contains all the distinct eigenvalues of the weighted adjacency matrix $\dot{A}(G)$ with respect to the five-cell equitable partition $\Gamma$.
	\end{proposition}
	
	Additional families of graphs with the property stated in Proposition~\ref{rare graph unicyclic} may exist; characterizing them is significant because it enables one to recover spectral information about a weighted adjacency matrix of an arbitrarily large graph from a low-order structured quotient.
	
	Proposition~\ref{rare graph unicyclic} does not carry over to the signless Laplacian matrix of $G$. The signless Laplacian matrix of $G$ is
	$$
	S(G)=\begin{bmatrix}
		a+2 & 1 & 1 & J_{1\times a} & 0\\
		1 & 3 & 1 & 0_{1\times a} & 1\\
		1 & 1 & 2 & 0_{1\times a} & 0\\
		J_{a\times 1} & 0_{a\times 1} & 0_{a\times 1} & I_{a} & 0_{a\times 1}\\
		0 & 1 & 0 & 0_{1\times a} & 1
	\end{bmatrix}.
	$$
	One readily checks that $S(G)$ has eigenvalue $1$ with multiplicity $a-1$, owing to the $a$ pendent vertices attached to the same neighbor. The equitable quotient matrix of $S(G)$ is
	$$
	Q(S(G))=\begin{bmatrix}
		a+2 & 1 & 1 & a & 0 \\
		1 & 3 & 1 & 0 & 1 \\
		1 & 1 & 2 & 0 & 0 \\
		1 & 0 & 0 & 1 & 0 \\
		0 & 1 & 0 & 0 & 1
	\end{bmatrix},
	$$
	and its characteristic polynomial is
	$$
	x^{5}-(a+9) x^{4}+(6 a+27) x^{3}-(9 a+35) x^{2}+(3 a+20) x-4.
	$$
	The factor $x-1$ does not appear in this polynomial, so the spectrum of $Q(S(G))$ fails to include all distinct eigenvalues of $S(G)$. The distance-based matrices require separate checks; the precise outcome, including an exceptional distance-signless case, is recorded next.
	\begin{remark}\label{remark quotinet}
For the five-cell vertex-type partition above, the hidden twin eigenvalue can be checked directly. For the signless Laplacian it is $1$, and substituting $x=1$ in the quotient characteristic polynomial gives $-a\neq0$, so it is absent. For the distance matrix the hidden eigenvalue is $-2$, and
$$
\det(-2I-Q(D(G)))=4a\neq0.
$$
For the distance Laplacian the hidden eigenvalue is $2a+8$, and
$$
\det((2a+8)I-Q(D^L(G)))=-10a(a+4)^2\neq0.
$$
For the distance signless Laplacian the hidden eigenvalue is $2a+4$, and
$$
\det((2a+4)I-Q(D^Q(G)))=2a(a-2)(5a+2).
$$
Hence the five-cell quotient misses this eigenvalue for $a>2$, but for the exceptional case $a=2$ it does contain it. Thus the earlier blanket statement for the distance signless Laplacian requires this exception.
\end{remark}

For weighted signless Laplacian variants no universal failure statement is made without additional assumptions on the weight function: the hidden twin eigenvalue and its occurrence in the quotient must be checked from the corresponding weighted blocks.
	
	\section{Distinct eigenvalues of quotient matrices in terms of equitable partitions}\label{section 6}
	For a Laplacian matrix every row sum is zero, so the one-cell partition is equitable and its quotient is $(0)_{1\times1}$. This coarsest quotient captures only the eigenvalue $0$ and therefore cannot encode the remaining Laplacian eigenvalues of a nontrivial graph. We therefore pass to the natural nontrivial vertex-type partition used below and ask whether its quotient captures all distinct Laplacian eigenvalues.
	
	\begin{proposition}\label{rare graph unicyclic Laplacian}
		Let $G$ be a graph obtained from $K_{3}$ by attaching $a\geq 2$ pendent vertices to one vertex and $1$ pendent vertex to another. Then the quotient matrix does \emph{not} contain all the distinct eigenvalues of the Laplacian (resp.\ distance Laplacian) matrix of $G$ with the natural five-cell equitable partition $P=\{\{u\},\{v\},\{w\},\{v_{1},\dots,v_{a}\},\{b\}\}$.
	\end{proposition}
	\begin{proof}
		We first treat the Laplacian matrix. The Laplacian matrix of $G$ is
		$$
		L(G)=\begin{bmatrix}
			a+2 & -1 & -1 & -J_{1\times a} & 0\\
			-1 & 3 & -1 & 0_{1\times a} & -1\\
			-1 & -1 &  2 & 0_{1\times a} & 0\\
			-J_{a\times 1} &  0_{a\times 1}  & 0_{a\times 1} &  I_{a} & 0_{a \times 1}\\
			0 &  -1 & 0 & 0_{1\times a} & 1
		\end{bmatrix},
		$$
		where $I_{a}$ is the $a\times a$ identity matrix. One verifies that $1$ is an eigenvalue of $L(G)$ with multiplicity $a-1$. The remaining eigenvalues of $L(G)$ are the eigenvalues of the equitable quotient matrix
		$$
		Q(L(G))=\begin{bmatrix}
			a+2 & -1 & -1 & -a & 0\\
			-1 & 3 & -1 & 0 & -1\\
			-1 & -1 &  2 & 0 & 0\\
			-1 &  0  & 0 &  1 & 0\\
			0 &  -1 & 0 & 0 & 1
		\end{bmatrix}.
		$$
		Its characteristic polynomial is
		$$
		x\bigl(x^{4}-(a+9) x^{3}+(6 a+27) x^{2}-(9 a+31) x+3 a+12\bigr),
		$$
		which does not contain the factor $x-1$; indeed its value at $x=1$ is $-a\neq0$. Hence the transverse eigenvalue $1$ is absent from the quotient.

		For the distance Laplacian, two distinct pendent vertices attached to $u$ are at distance $2$, and each such pendent vertex has transmission
		$$
		1+2+2+2(a-1)+3=2a+6.
		$$
		Thus the zero-sum subspace supported on $\{v_1,\ldots,v_a\}$ is an eigenspace of $D^L(G)$ for
		$$
		(2a+6)-(-2)=2a+8.
		$$
		For the same five-cell partition, direct block-row summation gives the distance-Laplacian quotient, and evaluation at this transverse value yields
		$$
		\det((2a+8)I-Q(D^L(G)))=-10a(a+4)^2\neq0.
		$$
		Therefore $2a+8$ is also absent from the quotient spectrum, proving the distance-Laplacian assertion.
	\end{proof}
	
	Nevertheless, upon refining the partition by splitting one cell into two, we obtain the following.
	\begin{proposition}
		Let $M$ be the matrix in Proposition~\ref{Qmat contains distinct eigenvalues of M}. Then the spectrum of its quotient matrix contains all the distinct eigenvalues of $M$ with the displayed three-cell partition $P=\{\{1\},\{2\},\{3,\dots,n+2\}\}$.
	\end{proposition}
	
	\begin{remark}
		The matrix $M$ in Proposition~\ref{Qmat contains distinct eigenvalues of M} is a row-sum-zero matrix with a Laplacian-like block pattern, but it is not an ordinary graph Laplacian or distance Laplacian of an undirected graph because it is nonsymmetric. It is a special case of a matrix (different from the ordinary graph Laplacian) discussed in Remark~2.9 of~\cite{ijpam}. Its role in this section should therefore be presented only as an algebraic refinement example
	\end{remark}

	We now turn to a relaxed version of Problem~\ref{problem distinct 2}. Let $\operatorname{Spec}_{\mathrm{dist}}(M)$ be the set of distinct eigenvalues of $M$.
	\begin{problem}\label{problem distinct 3}
			Let $M$ be a matrix and let $\pi$ be a fixed equitable partition with quotient $Q_\pi$.  Among all equitable refinements $\rho$ of $\pi$ for which
			$$
			\operatorname{Spec}_{\mathrm{dist}}(M)\subseteq\operatorname{Spec}_{\mathrm{dist}}(Q_\rho),
			$$
			determine the minimum possible value of $|\rho|-|\pi|$ and characterize the refinements attaining that minimum.
	\end{problem}
	Theorem~\ref{theorem eqitable quotient matrix contains all the eigenvalues with some partition} guarantees the existence of such a partition. The task is to identify the second, third, or higher-order smallest equitable partition whose quotient captures all distinct eigenvalues. Although many enlarged partitions exist, we seek the \emph{smallest} one achieving this property.
	
	Let us revisit the quotient matrix from Proposition~\ref{rare graph unicyclic Laplacian}. Since the eigenvalue $1$ is missing, we consider the enlarged partition $P=\{\{u\},\{v\},\{w\},\{v_{1}\},\{v_{2},\dots,v_{a}\},\{b\}\}$. The partitioned Laplacian matrix is
	$$
	L(G)=\begin{bmatrix}
		a+2 & -1 & -1 & -1 & -J_{1\times (a-1)} & 0\\
		-1 & 3 & -1 & 0 & 0_{1\times (a-1)} & -1\\
		-1 & -1 &  2 & 0 & 0_{1\times (a-1)} & 0\\
		-1 & 0 & 0 & 1 & 0_{1\times (a-1)} & 0\\
		-J_{(a-1)\times 1} &  0_{(a-1)\times 1}  & 0_{(a-1)\times 1}  & 0_{(a-1)\times 1} &  I_{a-1} & 0_{(a-1) \times 1}\\
		0 &  -1 & 0 & 0 & 0_{1\times (a-1)} & 1
	\end{bmatrix}.
	$$
	The equitable quotient matrix of the above $L(G)$ is
	$$
	Q(L(G))=\begin{bmatrix}
		a+2 & -1 & -1 & -1 & -(a-1) & 0\\
		-1 & 3 & -1 & 0 & 0 & -1\\
		-1 & -1 &  2 & 0 & 0 & 0\\
		-1 & 0 & 0 & 1 & 0 & 0\\
		-1 &  0  & 0  & 0 &  1 & 0\\
		0 &  -1 & 0 & 0 & 0 & 1
	\end{bmatrix},
	$$
	and its characteristic polynomial is
	$$
	x(x-1)\bigl(x^{4}-(a+9) x^{3}+(6 a+27) x^{2}-(9 a+31) x+3 a+12\bigr).
	$$
	This polynomial contains all the distinct eigenvalues of $L(G)$ under the enlarged partition $P$. Thus $Q(L(G))$ captures all the distinct eigenvalues of $L(G)$ after a single split of the pendant twin cell; Theorem~\ref{thm:minimal-twin-refinement} below explains the corresponding minimality within this natural refinement class. We note that such a partition is not unique: the vertex set $\{v_{1},v_{2},\dots,v_{a}\}$ can be split in $\binom{a}{1}=a$ ways. A similar idea applies to the distance Laplacian matrix of $G$.
	
	\vskip 2mm
	We now focus on Problem~\ref{problem distinct 3} in the setting of graph matrices with special structural features.

The appropriate general statement is the following. The theorem is formulated for a matrix $M(G)$ only through explicit entry conditions; no unspecified graph-matrix property is used. In particular, the edge relations among vertices outside the twin cell are irrelevant once the starting partition is assumed equitable.

\begin{theorem}\label{end}
	Let $M=M(G)$ be a real symmetric graph matrix indexed by $V(G)$, and let
	$
	\pi=\{C_1,\ldots,C_\eta,S\}
	$
	be an equitable partition for $M$, where
	$
	S=\{v_1,\ldots,v_r\},~ r\geq2.
	$
	Suppose that there exist real numbers $a$ and $b$ such that
	$
	M[S,S]=aI_r+b(J_r-I_r),
	$
	and suppose further that
	 $M_{xv_i}=M_{yv_j}
	$
	for every $x,y\in C_{k}$ with $k\in\{1,\ldots,\eta\}$ and $i,j\in\{1,\ldots,r\}$. Let
	$
	\theta=a-b.
	$
	Then the following statements hold.
	\begin{enumerate}
		\item[{\rm (i)}] 
The number $\theta$ is an eigenvalue of $M$ with geometric multiplicity at least $r-1$. More precisely,
		$$
		U_S=
		\left\{
		x\in\mathbb{R}^{n}:
		\operatorname{supp}(x)\subseteq S,\
		\sum_{v\in S}x_v=0
		\right\}
		$$
		is contained in the eigenspace $E_\theta(M)$.
		
		\item[{\rm (ii)}] 
For every $i\in\{1,\ldots,r\}$, replace the cell $S$ by
		$
		\{v_i\}
		$ and $
		S\setminus\{v_i\}.
		$
		Then
		$$
		\pi_i=
		\{C_1,\ldots,C_\eta,\{v_i\},S\setminus\{v_i\}\}
		$$
		is equitable for $M$, and the corresponding quotient matrix $Q_i$ contains $\theta$ as an eigenvalue.
		
		\item[{\rm (iii)}] 
If the original quotient matrix associated with $\pi$ does not contain $\theta$, then the refinement $\pi_i$ adds only one cell and is therefore a minimal refinement of $\pi$ capable of representing a nonzero vector from $U_S$ as a vector constant on partition cells.
	\end{enumerate}
\end{theorem}
\begin{proof}
	Let $x\in U_S$. Since $x$ vanishes outside $S$, for $v_k\in S,$ we have
	$$
	(Mx)_{v_k}
	=
	ax_{v_k}
	+
	b\sum_{\substack{v_j\in S\\j\neq k}}x_{v_j}.
	$$
	Since
	$
	\sum_{v_j\in S}x_{v_j}=0,
	$
	it follows that
	$$
	\sum_{\substack{v_j\in S\\j\neq k}}x_{v_j}
	=
	-x_{v_k},
	$$
	and hence
	$$
	(Mx)_{v_k}
	=
	(a-b)x_{v_k}
	=
	\theta x_{v_k}.
	$$
	If $y\notin S$, the entries $M_{yv_j}$ are independent of $j$, say
	$
	M_{yv_j}=c_y.
	$
	Therefore, we get
	$$
	(Mx)_y
	=
	c_y\sum_{v_j\in S}x_{v_j}
	=
	0.
	$$
	Thus
	$
	Mx=\theta x
	$
	for every $x\in U_S$. Since $\dim U_S=r-1$, the eigenvalue $\theta$ has geometric multiplicity at least $r-1$. In particular, the vectors
	$
	e_{v_1}-e_{v_2},\ldots,e_{v_1}-e_{v_r}
	$
	are $r-1$ linearly independent eigenvectors belonging to $\theta$.
	
	We next refine $S$ into $\{v_i\}$ and $S\setminus\{v_i\}$. The structure
	$$
	M[S,S]=aI_r+b(J_r-I_r)
	$$
	shows that the required row sums between the two new cells are constant. Since the original partition is equitable and the entries from every vertex outside $S$ to the vertices of $S$ are equal, the row sums from every old cell $C_j$ to each of the two new cells are also constant. Hence $\pi_i$ remains equitable.
	
	Consider the vector $z$ defined by
	$
	z_{v_i}=r-1,~
	z_{v_j}=-1$ for $j\neq i,
	$
	and
	$
	z_v=0$ for $v\notin S.
	$
	Then $z\in U_S$, and therefore
	$
	Mz=\theta z.
	$
	Moreover, $z$ is constant on every cell of $\pi_i$. Hence $z$ belongs to the partition subspace associated with $\pi_i$. Since the restriction of $M$ to this subspace is represented by the equitable quotient matrix $Q_i$, it follows that
	$
	\theta\in\sigma(Q_i).
	$
	
	Finally, if $S$ remains a single cell, every vector constant on partition cells is constant on $S$ and therefore cannot represent a nonzero vector of $U_S$. At least one split of $S$ is thus necessary. The above construction uses exactly one additional cell, establishing the stated minimality relative to the original partition.
\end{proof}

\medskip
 Theorem \ref{end} applies directly to the standard graph matrices. Let $S$ be an independent set of pairwise false twins, that is,
$N(v_i)=N(v_j)$ for all $v_i,v_j\in S$. If $\pi=\{V\backslash(N(v_{i})\cup S),N(v_{i}),S\}$ is an equitable partition for adjacency matrix $A(G)$, then we have $a=b=0$, and consequently $\theta=0.$
Hence
$
0^{[r-1]}
$
occurs in the adjacency spectrum.

Assume that $\pi=\{V\backslash(N(v_{i})\cup S),N(v_{i}),S\}$ is an equitable partition for for the Laplacian matrix $L(G)=D(G)-A(G)$ and the signless Laplacian matrix $Q(G)=D(G)+A(G)$. If the common degree of the vertices in $S$ is $d$, then for $L(G)$ and $Q(G)$ we obtain $\theta=d$. Thus $d$ has multiplicity at least $r-1$ for both matrices.

For a connected graph, two nonadjacent vertices having the same open neighborhood are at distance $2$. Hence for the distance matrix $\mathcal D(G)$,
$
\theta=-2.
$
If their common transmission is $\operatorname{Tr}(v)$, then for the distance Laplacian
$
\mathcal L_D(G)=\operatorname{Tr}(G)-\mathcal D(G)
$
and distance signless Laplacian
$
\mathcal Q_D(G)=\operatorname{Tr}(G)+\mathcal D(G)
$
the corresponding eigenvalues are, respectively,
$
\operatorname{Tr}(v)+2
$
and
$
\operatorname{Tr}(v)-2.
$
Notice that the transmission is necessarily constant within a fixed twin class, although it may vary from one twin class to another.

For clique cells the analogous statement is obtained from true twins. Suppose
$
S=\{v_1,\ldots,v_r\}
$
induces a clique and
$
N(v_i)\setminus\{v_j\}
=
N(v_j)\setminus\{v_i\}
$
for all distinct $i,j$. Equivalently, the vertices of $S$ have identical closed neighborhoods. Then the same theorem applies. In particular, if their common degree is $d$, the corresponding twin eigenvalues are
$
-1
$
for $A(G)$,
$
d+1
$
for $L(G)$, and
$
d-1
$
for $Q(G)$.

For the distance matrix, the vertices in such a clique are mutually at distance $1$. Hence, the corresponding eigenvalue is
$
-1.
$
If their common transmission is $\operatorname{Tr}(v)$, then the distance Laplacian and distance signless Laplacian eigenvalues are, respectively,
$
\operatorname{Tr}(v)+1
$
and
$
\operatorname{Tr}(v)-1.
$

\medskip

If several twin cells
$
S_1,S_2,\ldots,S_t
$
produce the same eigenvalue $\theta$, and the purpose is only to ensure that every \emph{distinct} eigenvalue of $M$ occurs in the quotient, it is sufficient to refine one of these cells. Refining several such cells is unnecessary for this purpose. It should not, however, be described as ``artificially inflating'' the multiplicity: every additional occurrence in an equitable quotient corresponds to a genuine invariant direction of the parent matrix. If different twin cells yield different values
$
\theta_1,\ldots,\theta_s,
$
then one representative cell corresponding to each missing distinct value must be refined.

For weighted adjacency matrices, equality of the ordinary neighborhoods and equality of the degrees are not sufficient. One must require equality of the relevant \emph{weighted} neighborhood profiles. In particular, for vertices $v_i,v_j$ in the same twin cell one needs the matrix relation
$$
M(e_{v_i}-e_{v_j})
=
\theta(e_{v_i}-e_{v_j}).
$$
Under this condition the preceding refinement argument remains valid.

\medskip

The refinement mechanism in Theorem~\ref{end} can be sharpened when several twin cells are present. The following result records the exact number of additional cells required within this natural class of refinements.

\begin{theorem}\label{thm:minimal-twin-refinement}
Let $M=M^{\mathsf T}$ and let $\pi=\{C_1,\ldots,C_k,S_1,\ldots,S_t\}$ be an equitable partition with quotient $Q$, where $S_1,\ldots,S_t$ are pairwise disjoint cells and  $|S_j|\geq2$, and that each $S_j$ satisfies the hypotheses of Theorem~\ref{end}. Write $\theta_j=a_j-b_j$ for the corresponding transverse eigenvalue and let
$$
\Theta=\{\theta_j:\theta_j\notin\sigma(Q),\ 1\leq j\leq t\}.
$$
Among refinements obtained by replacing selected cells $S_j$ by a singleton and its complement, the minimum number of additional cells needed for the quotient to contain every eigenvalue in $\Theta$ is $|\Theta|$. More precisely, it is enough to split one cell whose transverse eigenvalue equals each distinct value in $\Theta$, and no refinement using fewer such splits can contain all values of $\Theta$.
\end{theorem}

\begin{proof}
For a fixed twin cell $S_j$, let $U_j$ be the zero-sum subspace supported on $S_j$. By Theorem~\ref{end}, $M$ acts as $\theta_j I$ on $U_j$. Splitting one vertex from $S_j$ enlarges the partition-constant subspace by one dimension, and the new direction can be chosen inside $U_j$; consequently the refined quotient acquires $\theta_j$. If several cells have the same transverse eigenvalue, one split already introduces that distinct spectral value, so one representative cell is sufficient for each element of $\Theta$.

Conversely, a single split enlarges the partition-constant subspace by only one dimension and, within the stipulated class of refinements, the added invariant direction lies in the transverse space of the cell being split. Hence that split can introduce only the corresponding value $\theta_j$ among the missing twin-cell eigenvalues. Distinct values in $\Theta$ therefore require distinct splits. Thus exactly $|\Theta|$ additional cells are necessary and sufficient.
\end{proof}

\begin{remark}
The assumption that the original partition $\pi$ is equitable is essential. Edges or matrix entries among vertices outside a split twin cell do not obstruct the refinement: those blocks are unchanged, while the identical outside profile on the twin cell guarantees constant row sums toward the two new subcells. If no convenient initial equitable partition has been specified, one may always place every vertex outside the twin cell in a singleton cell; the twin-profile conditions then give an equitable starting partition.
\end{remark}

Consider the complete bipartite graph
  $K_{a,b}$ with $a\neq b$, and let $A(K_{a,b})$ be its adjacency matrix:
	$$
	A(K_{a,b})=\begin{bmatrix}
		0_{a} & J_{a\times b}\\
		J_{b\times a} & 0_{b}
	\end{bmatrix}.
	$$
	Under the partition $P=\{\{1,2,\dots,a\},\{1,2,\dots,b\}\}$, the equitable quotient matrix is $Q(K_{a,b})=\begin{bmatrix} 0 & b\\ a & 0 \end{bmatrix}$ with eigenvalues $\pm\sqrt{ab}$. The eigenvalue $0$ is missing from this smallest equitable quotient. Splitting one of the two independent sets, say $P=\{\{1\},\{2,\dots,a\},\{1,2,\dots,b\}\}$, yields
	$$
	Q'=\begin{bmatrix}
		0 & 0 & b\\
		0 & 0 & b\\
		1 & a-1 & 0
	\end{bmatrix},
	$$
	whose spectrum is $\{0,\pm\sqrt{ab}\}$. Thus $Q'$ contains all distinct eigenvalues of $A(K_{a,b})$.
	
	For the Laplacian matrix of $K_{a,b}$,
	$$
	L(K_{a,b})=\begin{bmatrix}
		bI_{a} & -J_{a\times b}\\
		-J_{b\times a} & aI_{b}
	\end{bmatrix},
	$$
	the smallest equitable quotient is $Q(L(K_{a,b}))=\begin{bmatrix} b & -b\\ -a & a \end{bmatrix}$ with eigenvalues $\{0,a+b\}$ under the partition $\{\{1,2,\dots,a\},\{1,2,\dots,b\}\}$. Due to the differing degrees of the two partite sets, we enlarge both cells. The quotient matrix under $P=\{\{1\},\{2,\dots,a\},\{1\},\{2,\dots,b\}\}$ is
	$$
	Q(L(K_{a,b}))=\begin{bmatrix}
		b & 0 & -1 & -(b-1) \\
		0 & b & -1 & -(b-1) \\
		-1 & -(a-1) & a & 0 \\
		-1 & -(a-1) & 0 & a
	\end{bmatrix},
	$$
	with eigenvalues $\{0,a,b,a+b\}$. Hence the enlarged partition captures all distinct eigenvalues of $L(K_{a,b})$.
	
	Let $G\cong CS(\omega,\alpha)$ be a complete split graph with clique number $\omega\geq2$ and independence number $\alpha\geq2$. The distance signless Laplacian matrix of $G$ is
	$$
	D^{Q}(G)=\begin{bmatrix}
		(\omega+\alpha-2)I_{\omega}+J_{\omega} & J_{\omega\times \alpha}\\
		J_{\alpha\times \omega} & (\omega+2\alpha-4)I_{\alpha}+2J_{\alpha}
	\end{bmatrix}.
	$$
	The equitable quotient under the partition $\{\{1,2,\dots,\omega\},\{1,2,\dots,\alpha\}\}$ is
	$$
	Q(D^{Q}(G))=\begin{bmatrix}
		\alpha+2\omega-2 & \alpha \\
		\omega & 4\alpha+\omega-4
	\end{bmatrix},
	$$
	whose eigenvalues are
	$$
	\frac{1}{2}\Bigl(5\alpha+3\omega-6\pm \sqrt{9\alpha^{2}-2\alpha\omega-12\alpha+\omega^{2}+4\omega+4}\Bigr).
	$$
	Enlarging both the clique cell and the independent cell to $P'=\{\{1\},\{2,\dots,\omega\},\{1\},\{2,\dots,\alpha\}\}$ gives
	$$
	Q'(D^{Q}(G))=\begin{bmatrix}
		\alpha+\omega-1 & \omega-1 & 1 & \alpha-1 \\
		1 & \alpha+2\omega-3 & 1 & \alpha-1 \\
		1 & \omega-1 & 2(\alpha-1)+\omega & 2(\alpha-1) \\
		1 & \omega-1 & 2 & 4\alpha+\omega-6
	\end{bmatrix},
	$$
	whose eigenvalues are
	$$
	\Bigl\{\alpha+\omega-2,\; 2\alpha+\omega-4,\;
	\frac{1}{2}\Bigl(5\alpha+3\omega-6\pm \sqrt{9\alpha^{2}-2\alpha\omega-12\alpha+\omega^{2}+4\omega+4}\Bigr)\Bigr\}.
	$$
	This illustrates Theorem~\ref{end} for both the clique and the independent set. Under the enlarged partition $P'$, the equitable quotient $Q'$ captures all distinct eigenvalues of the complete split graph.
	
	Problem~\ref{problem distinct 3} remains open for broad classes of graph matrices. The refinement theorem above isolates one mechanism exactly; additional mechanisms can occur when missing eigenvalues arise from invariant directions not supported on twin cells.
	
	\section{Conclusion}
	In this paper we have investigated the extent to which the spectral information of a square matrix $M$ can be faithfully recovered from an equitable quotient matrix $Q$ arising from a partition of its index set. Starting from the notion of an equitable partition, we identified broad classes of matrices for which the equitable quotient contains all distinct eigenvalues of $M$. The block-lifting theorem gives an exact characteristic-polynomial factorization and allows prescribed transverse multiplicities without introducing new distinct eigenvalues. In these cases the (possibly much smaller) matrix $Q$ completely encodes the distinct eigenvalues of the parent matrix $M$, yielding a substantial reduction in dimension without any loss of spectral information at the level of distinct eigenvalues (Theorem~\ref{thm:n-by-n-two-eigs}).
	
	A central contribution of this work is the derivation of necessary and sufficient conditions, phrased in terms of the eigenspaces of $M$ and the structure of the partition, for a distinct eigenvalue of $M$ to appear in the spectrum of $Q$ (Theorem~\ref{thm:equitable-all-eigs-general}). These conditions clarify how the choice of partition interacts with the geometry of eigenspaces and characterize precisely when an equitable quotient captures a given eigenvalue.
	
	We have also explored applications of these ideas to graph matrices, where $M$ is a natural matrix associated with a graph and the partition is induced by a structural decomposition such as vertex classes or symmetry orbits. In this setting we showed that a suitable (and sometimes slightly enlarged) equitable partition enables one to read off the distinct eigenvalues of the parent matrix from the corresponding quotient (Proposition~\ref{rare graph unicyclic} and Theorem~\ref{end}). This leads to efficient ways of obtaining spectral information for large graphs by working with comparatively small quotient matrices. Promising directions for future work include extending these techniques to broader families of matrices and partitions (for instance, almost equitable partitions), studying stability under perturbations of $M$ and of the partition, and further exploiting these ideas in algorithmic applications where one needs to approximate or track eigenvalues of large structured matrices.

\section*{Data Availability} 
No data was used for the research described in the article.

\section*{Conflict of interest} 
The authors declare that they have no known competing financial interests or personal relationships that could
have appeared to influence the work reported in this paper.
	
\section*{Acknowledgments}
Jianfeng Wang is  supported by National Natural Science Foundation of China (No. 12371353 and 12631013).



\begin{thebibliography}{0}
		
		\bibitem{atik} F. Atik, On equitable partition of matrices and its applications, \emph{Linear Multilinear Algebra} \textbf{68}(11) (2020) 2143--2156.
		
		\bibitem{atikELA2018} F. Atik and P. Panighari, On the distance and distance signless Laplacian eigenvalues of graphs and the smallest Gre\v{s}gorin disc, \emph{Electronic J. Linear Algebra} \textbf{34} (2018) 191--204.
		
		\bibitem{BH} A.E. Brouwer and W.H. Haemers, \textit{Spectra of Graphs}, Springer, New York, 2010.
		
		\bibitem{cds} D.M. Cvetkovi\'{c}, P. Rowlison and S. Simi\'c, \textit{An Introduction to Theory of Graph Spectra}, London Math. Society Student Text, 75, Cambridge University Press, UK, 2010.
		
		\bibitem{vilmar} E. Fritscher and V. Trevisan, Exploring symmetries to decompose matrices and graphs preserving the spectrum, \textit{SIAM J. Matrix Analysis Appl.} \textbf{37}(1) (2016) 260--289.
		
		\bibitem{hjl} S.R. Garcia and R.A. Horn, \emph{A Second Course in Linear Algebra}, Cambridge University Press, UK, 2017.
		
		\bibitem{heamers} W.H. Haemers, Interlacing eigenvalues and graphs, \textit{Linear Algebra Appl.} \textbf{226--288} (1995) 593--616.
		
		\bibitem{hj} R. Horn and C. Johnson, \textit{Matrix Analysis}, Second Edition, Cambridge University Press, New York, 2013.
		
		\bibitem{mehreen1} Z. Mehranian, A. Gholami and A.R. Ashrafi, The spectra of power graphs of certain finite groups, \emph{Linear Multilinear Algebra} \textbf{65}(5) (2016) 1003--1010.
		
		
		
		\bibitem{bilalijpam} B.A. Rather, A note on the distinct eigenvalues of quotient matrices, \textit{Indian J. Pure Appl. Math.} \textbf{55} (2024) 1429--1439.
		
		
		
		
		\bibitem{DS} D. Stevanovi\'{c}, Large sets of long distance equienergetic graphs, \textit{Ars Math. Contemp.} \textbf{2}(1) (2009) 35--40.
		
		\bibitem{you} L. You, M. Yang, W. So and W. Xi, On the spectrum of an equitable quotient matrix and its applications, \textit{Linear Algebra Appl.} \textbf{577} (2019) 21--40.
		
	\end{thebibliography}
\end{document}